\documentclass{amsart}
\usepackage{amsmath,mathrsfs,cite}

\usepackage[ps2pdf,colorlinks=true,urlcolor=blue,
citecolor=red,linkcolor=blue,linktocpage,pdfpagelabels,bookmarksnumbered,bookmarksopen]{hyperref}
\usepackage[english]{babel}

\usepackage[left=2.95cm,right=2.95cm,top=2.8cm,bottom=2.8cm]{geometry}

\numberwithin{equation}{section}

\newcommand{\N}{{\mathbb N}}

\newcommand{\Q}{{\mathbb Q}}
\newcommand{\R}{{\mathbb R}}
\newcommand{\T}{{\mathbb T}}

\newcommand{\eps}{\varepsilon}

\newcommand{\dom}[1]{{\rm dom}(#1)}

\renewcommand{\theta}{\vartheta}

\numberwithin{equation}{section}
\newtheorem{theorem}{Theorem}[section]
\newtheorem{proposition}[theorem]{Proposition}

\newtheorem{remark}[theorem]{Remark}
\newtheorem{corollary}[theorem]{Corollary}
\newtheorem{definition}[theorem]{Definition}
\theoremstyle{definition}

\title{Symmetry in variational principles and applications}
\author{Marco Squassina}
\address{Dipartimento di Informatica
\newline\indent
Universit\`a degli Studi di Verona
\newline\indent
C\'a Vignal 2, Strada Le Grazie 15, I-37134 Verona, Italy}
\email{marco.squassina@univr.it}

\thanks{Research supported by PRIN (2007): {\em Metodi Variazionali e Topologici
nello Studio di Fenomeni non Lineari}}

\begin{document}
	

\subjclass[2010]{35A15; 35B06; 58E05; 65K10}

\keywords{Variational principles, symmetry, Palais-Smale sequences, weak and strong slope.}

\begin{abstract}
We formulate symmetric versions of classical variational principles. Within the framework 
of non-smooth critical point theory, we detect Palais-Smale sequences with
additional second order and symmetry information. We discuss applications to  
PDEs, fixed point theory and geometric analysis.
\end{abstract}
\maketitle




\section{Introduction}
One of the most powerful contributions of the last decades in calculus of variations and nonlinear analysis is
surely given by Ekeland's variational principle for lower semi-continuous 
functionals on complete metric spaces \cite{ekeland1,ekeland2}, arisen
in the context of convex analysis.
We refer to \cite{aubin-ekeland,borwzhu,defig,ekeland2,ghoubook} where a multitude of applications in different 
fields of analysis is carefully discussed. In a recent note \cite{squass} the author 
has proved a version of the principle in Banach spaces which provides
{\em almost symmetric} almost critical points, provided that the functional satisfies
a rather mild symmetry condition. Roughly speaking, if $(X,\|\cdot\|)$ is a Banach space which is continuously embedded
into a space $(V,\|\cdot\|_V)$ with suitable properties and $f:X\to\R\cup\{+\infty\}$ is a lower 
semi-continuous bounded below functional which does not increase by polarization, then for all $\eps>0$
there is $u_\eps\in X$ with
$$
\|u_\eps-u^*_\eps\|_V<\eps,\qquad f(u_\eps)\leq \inf f+\eps^2,\qquad
f(\xi)\geq f(u_\eps)-\eps \|\xi-u_\eps\|\quad\forall \xi\in X,
$$
where the symmetrization $*$ is defined in an abstract framework, which 
reduces to the classical notions in concrete functional spaces, such as in $L^p(\Omega)$ and in $W^{1,p}_0(\Omega)$ spaces, 
being $\Omega$ either a ball in $\R^N$ or the whole $\R^N$.
Possessing almost symmetric points is very useful in applications not only to find symmetric cluster points, but also in order
to facilitate the strong convergence of the sequence $(u_\eps)$ via suitable compact embeddings enjoyed 
by spaces $X_*$ of symmetric functions of $X$ \cite{lionsym,strauss,willem}.
The aim of the present manuscript is that of giving a rather complete range of abstract 
results in this direction furnishing also applications to calculus of variations, fixed point theory and geometry of Banach spaces.
\vskip2pt
\noindent
The plan of the paper is as follows.
In Section~\ref{varprinciples}, we will state symmetric versions 
of Ekeland \cite{ekeland1}, Borwein-Preiss \cite{borweinpreiss} and
Deville-Godefroy-Zizler \cite{deville} principles, free or constrained, as well as versions for the Ekeland's principle with weights,
in the spirit of Zhong's result \cite{zhong} (see Theorems~\ref{mainthm}, \ref{ekelcor}, 
\ref{ekelcorII}, \ref{ekelcorIII}, \ref{ekelcor4}, \ref{ekelcorV}, \ref{mainthm-deville} and \ref{symzhong}).
Furthermore, in the framework of the
non-smooth critical point theory developed in \cite{dm}, we will detect suitable Palais-Smale sequences $(u_h)$ with
respect to the notion of weak slope whose elements $u_h$ become more and more symmetric, $u_h\sim u_h^*$, as $h$ gets large,
and satisfy a second order information, in terms of a quantity $w\mapsto Q_{u_h}(w)$,
introduced in \cite{bartschdeg}, that
plays the r\v ole of the quadratic form $w\mapsto f''(u_h)(w)^2$ for functions of class $C^2$
(see Theorem~\ref{maincor} as well as Corollary~\ref{corsym11}).
As pointed out by Lions in \cite{lionscmp}, this additional second order information can be 
very important to prove the strong convergence of $(u_h)\subset X$, 
in some physically meaningful situations. It would be interesting to obtain results
in the same spirit for mountain pass values in place of minimum values, as developed by Fang and Ghoussoub in \cite{fanghou} 
without symmetry information. In Section~\ref{remsymcoercps} a discussion upon the relationships between symmetry, coercivity
and Palais-Smale sequences is developed while in Section~\ref{minimaxclass} an application of the symmetric Ekeland principles
to get minimax type results is outlined. In Section \ref{applicationsprincip}, we discuss
some possible applications and implications of the abstract machinery formulated in Section~\ref{varprinciples}.
First, we find almost symmetric solutions, up to a perturbation, for two classes of nonlinear elliptic PDEs
associated with suitable energy functionals (see Theorems \ref{Eksymnonsm} and \ref{EksymnonsmOO}). 
Next, we get a symmetric version of
Caristi \cite{caristi} fixed point theorem and of a theorem due to Clarke \cite{clarke} 
(see Theorems~\ref{caristi} and \ref{clarkethm}) and we obtain some applications
in the geometry of Banach spaces, such as symmetric versions of Dane\u{s} Drop \cite{danes} and Flower Petal theorems \cite{penot}
(see Theorems \ref{drop} and \ref{petal}).

\section{Symmetric variational principles}
\label{varprinciples} 

\noindent
Let $X$, $V$ and $W$ be three real Banach spaces with $X\subseteq V\subseteq W$ 
and let $S\subseteq X$. 

\subsection{Abstract framework}
\label{polarizationsect}
Following \cite{vansch}, consider the following

\begin{definition}\label{abssym}
We consider two maps $*:S\to S$, $u\mapsto u^*$, 
the symmetrization map, and $h:S\times {\mathcal H}_*\to S$, $(u,H)\mapsto u^H$, the polarization map, ${\mathcal H}_*$ 
being a path-connected topological space. We assume that the following hold:
\begin{enumerate}
 \item $X$ is continuously embedded in $V$; $V$ is continuously embedded in $W$;
 \item $h$ is a continuous mapping;
\item for each $u\in S$ and $H\in {\mathcal H}_*$ it holds $(u^*)^H=(u^H)^*=u^*$ and $u^{HH}=u^H$;
\item there exists $(H_m)\subset {\mathcal H}_*$ such that, for $u\in S$, $u^{H_1\cdots H_m}$ converges
to $u^*$ in $V$;
\item for every $u,v\in S$ and $H\in {\mathcal H}_*$ it holds
$\|u^H-v^H\|_V\leq \|u-v\|_V$.
\end{enumerate}
Moreover, the mappings $h:S\times {\mathcal H}_*\to S$ and $*:S\to S$ can be extended to $h:X\times {\mathcal H}_*\to S$ and $*:X\to S$ by 
setting $u^H:=(\Theta(u))^H$ for every $u\in X$ and $H\in {\mathcal H}_*$ and
$u^*:=(\Theta(u))^*$ for every $u\in X$ respectively, where $\Theta:(X,\|\cdot\|_V)\to (S,\|\cdot\|_V)$ is
a Lipschitz function, of Lipschitz constant $C_\Theta>0$, such that $\Theta|_{S}={\rm Id}|_{S}$. 
\end{definition}

\noindent
The previous properties, in particular (4) and (5), and the definition of $\Theta$ easily yield:
\begin{equation}
	\label{contractivvv}
\forall u,v\in X,\,\forall H\in {\mathcal H}_*:\,\,\, \|u^H-v^H\|_V\leq C_\Theta\|u-v\|_V,\,\,\, \|u^*-v^*\|_V\leq C_\Theta\|u-v\|_V.
\end{equation}

\noindent
For the sake of completeness, we now recall some concrete notions.
\subsubsection{Concrete polarization}
\label{concresectpol}
A subset $H$ of $\R^N$ is called a polarizer if it is a closed affine half-space
of $\R^N$, namely the set of points $x$ which satisfy $\alpha\cdot x\leq \beta$
for some $\alpha\in \R^N$ and $\beta\in\R$ with $|\alpha|=1$. Given $x$ in $\R^N$
and a polarizer $H$ the reflection of $x$ with respect to the boundary of $H$ is
denoted by $x_H$. The polarization of a function $u:\R^N\to\R^+$ by a polarizer $H$
is the function $u^H:\R^N\to\R^+$ defined by
\begin{equation}
 \label{polarizationdef}
u^H(x)=
\begin{cases}
 \max\{u(x),u(x_H)\}, & \text{if $x\in H$} \\
 \min\{u(x),u(x_H)\}, & \text{if $x\in \R^N\setminus H$.} \\
\end{cases}
\end{equation}
The polarization $C^H\subset\R^N$ of a set $C\subset\R^N$ is 
defined as the unique set which satisfies $\chi_{C^H}=(\chi_C)^H$,
where $\chi$ denotes the characteristic function. 
The polarization $u^H$ of a positive function $u$ defined on $C\subset \R^N$
is the restriction to $C^H$ of the polarization of the extension $\tilde u:\R^N\to\R^+$ of 
$u$ by zero outside $C$. The polarization of a function which may change sign is defined
by $u^H:=|u|^H$, for any given polarizer $H$.

\subsubsection{Concrete symmetrization}
\label{concresectsym}
The Schwarz symmetrization of $C\subset \R^N$ is the unique open ball 
centered at the origin $C^*$ such that ${\mathcal L}^N(C^*)={\mathcal L}^N(C)$, being ${\mathcal L}^N$ the
Lebesgue measure on $\R^N$. If the measure of $C$ is zero, set $C^*=\emptyset$.
If the measure of $C$ is not finite, put $C^*=\R^N$. 
A measurable function $u$ is admissible for the Schwarz symmetrization if $u\geq 0$ and, for all $\eps>0$,
the measure of $\{u>\eps\}$ is finite. The Schwarz symmetrization
of an admissible $u:C\to\R^+$ is the unique $u^*:C^*\to\R^+$ such that,
for all $t\in\R$, it holds $\{u^*>t\}=\{u>t\}^*$. Considering the extension 
$\tilde u:\R^N\to\R^+$ of $u$ by zero outside $C$, it is $(\tilde u)^*|_{\R^N\setminus C^*}=0$ and
$u^*=(\tilde u)^*|_{C^*}$. Let ${\mathcal H}_*=\{H\in{\mathcal H}: 0\in H\}$ and let
$\Omega$ a ball in $\R^N$ or the whole space $\R^N$. Then $u=u^*$ if and only if $u=u^H$
for every $H\in {\mathcal H}_*$. Set either 
$X=W^{1,p}_0(\Omega)$, $S=W^{1,p}_0(\Omega,\R^+)$, $V=L^p\cap L^{p^*}(\Omega)$ with $h(u):=u^H$ and $*(u):=u^*$
or $X=S=W^{1,p}_0(\Omega)$, $V=L^p\cap L^{p^*}(\Omega)$ with $h(u):=|u|^H$ and $*(u):=|u|^*$.
In the first case $\Theta(u):=|u|$ defines a function from $(X,\|\cdot\|_V)$ to $(S,\|\cdot\|_V)$,
Lipschitz of constant $C_\Theta=1$, allowing to extend the definition of $h,*$ on $X=W^{1,p}_0(\Omega)$
by $h(u):=h(\Theta(u))$ and $*(u):=*(\Theta(u))$. In both cases properties (1)-(5) in 
Definition~\ref{abssym} are satisfied \cite{vansch}. 
\vskip5pt
\noindent
We now recall \cite[Corollary 3.1]{vansch} a 
useful result on the approximation of symmetrizations. 
The subset $S$ of $X$ in Definition~\ref{abssym} is considered
as a metric space with the metric $d$ induced by $\|\cdot\|$ on $X$. 
We assume that conditions (1)-(5) of Definition~\ref{abssym} are satisfied.

\begin{proposition}
	\label{mapJvS}
For all $\rho>0$ there exists a continuous mapping $\T_\rho:S\to S$ such that $\T_\rho u$ 
is built via iterated polarizations and $\|\T_\rho u-u^*\|_V<\rho$, for all $u\in S$.
\end{proposition}

\begin{remark}\rm
	\label{restriction}
If $S$ is the set involved in Definition~\ref{abssym}, assume that
$$
S'\subseteq S, \,\,\quad h(S'\times {\mathcal H}_*)\subseteq S', \,\,\quad *(S')\subseteq S'.
$$ 
Then $(S',X,V,h,*)$ satisfies conditions (1)-(5) of Definition~\ref{abssym} and Proposition~\ref{mapJvS}
holds for $S'$ in place of $S$. If $u\in X$, then one still defines $u^H:=(\Theta(u))^H$ 
and $u^*:=(\Theta(u))^*$ for all $u\in X$. Notice that $\Theta(u)=u$ for all $u\in S'$, since $S'\subseteq S$
and $\Theta|_{S}={\rm Id}|_{S}$.
\end{remark}

\subsection{Classical principles}

In the following, we recall a particular form, suitable for our purposes, 
of Borwein-Preiss's smooth variational principle \cite{borweinpreiss} for reflexive
Banach spaces endowed with a Kadek renorm (cf.~\cite[Theorems 2.6 and 5.2, and formula 5.4]{borweinpreiss}).  
We say that $X$ is endowed with a Kadec renorm $\|\cdot\|$, if the weak and 
norm topologies agree on the unit sphere of $X$. Such a norm indeed exists if $X$ is 
reflexive \cite{kadek}.

\begin{theorem}[Borwein-Preiss' principle]
	\label{BPprinc}
Assume that	$X$ is a reflexive Banach space, endowed with any Kadec norm $\|\cdot\|$.
Let $f:X\to\R\cup\{+\infty\}$ be a proper bounded below lower 
semi-continuous functional. Let $u\in X$, 
$\rho>0$, $\sigma>0$ and $p\geq 1$ be such that
$$
f(u)<\inf f+\sigma\rho^p.
$$
Then there exist $v\in X$ and $\eta\in X$ such that
\begin{enumerate}
\item[(a)] $\|v-u\|<\rho$;
\item[(b)] $\|\eta-u\|\leq\rho $;
\item[(c)]   $f(v)<\inf f+\sigma\rho^p$;   
\item[(d)]   $f(w)\geq f(v)+\sigma (\|v-\eta\|^p-\|w-\eta\|^p),$\quad\text{for all $w\in X$.}
\end{enumerate}
\end{theorem}

\noindent
The following is a symmetric version of Borwein-Preiss's smooth 
variational principle. 

\begin{theorem}[{Symmetric Borwein-Preiss' principle}]
	\label{mainthm}
Assume that	$X$ is reflexive Banach space, endowed with any Kadec norm $\|\cdot\|$.
Let $f:X\to\R\cup\{+\infty\}$ be a proper bounded below lower 
semi-continuous functional such that
\begin{equation}
	\label{assumptionpol-00}
\text{$f(u^H)\leq f(u),$\,\,\quad for all $u\in S$ and $H\in {\mathcal H}_*$}.
\end{equation}
Let $u\in S$, $\rho>0$, $\sigma>0$ and $p\geq 1$ be such that
\begin{equation}
	\label{startpoint}
f(u)<\inf f+\sigma\rho^p.
\end{equation}
Then there exist $v\in X$ and $\eta\in X$ such that
\begin{enumerate}
\item[(a)] $\|v-v^*\|_V< (K(C_\Theta+1)+1)\rho$; 
\item[(b)] $\|v-u\|< \rho+\|\T_\rho u-u\|$;
\item[(c)] $\|\eta-u\|\leq\rho +\|\T_\rho u-u\|$;
\item[(d)]   $f(v)<\inf f+\sigma\rho^p$;   
\item[(e)]   $f(w)\geq f(v)+\sigma (\|v-\eta\|^p-\|w-\eta\|^p),$\quad\text{for all $w\in X$.}
\end{enumerate}
Here $K>0$ denotes the continuity constant for the injection $X\hookrightarrow V$.
\end{theorem}
\begin{proof}
	Let $u\in S$, $\rho>0$, $\sigma>0$ and $p\geq 1$ be such that $f(u)<\inf f+\sigma\rho^p$.
	If $\T_\rho:S\to S$ is the mapping of Proposition~\ref{mapJvS},
	we set $\tilde u:=\T_\rho u\in S$. Then, by construction we have $\|\tilde u-u^*\|_V<\rho$
	and, in light of~\eqref{assumptionpol-00} and the property that $\tilde u$ is built 
	from $u$ through iterated polarizations, we obtain
	$f(\tilde u)<\inf f+\sigma\rho^p$. By Theorem~\ref{BPprinc}, there exist
	elements $v\in X$ and $\eta\in X$ with $\|\eta-\tilde u\|\leq \rho$, such that
	$f(v)<\inf f+\sigma\rho^p$, $\|v-\tilde u\|<\rho$ and
	$$
	f(w)\geq f(v)+\sigma (\|v-\eta\|^p-\|w-\eta\|^p),\quad\text{for all $w\in X$.}
	$$
	Hence (d) and (e) hold true. Taking into account the second inequality in \eqref{contractivvv}, if $K>0$ denotes the 
	continuity constant of the injection $X\hookrightarrow V$, we obtain
	\begin{equation}
		\label{diffsymmcontro}
	\|v-v^*\|_V \leq K(C_\Theta+1)\|v-\tilde u\|+\|\tilde u-u^*\|_V < (K(C_\Theta+1)+1)\rho,
	\end{equation}
	where we used the fact that $u^*=\tilde u^*$, in light of (3) of the
	abstract framework and, again, by the way $\tilde u$ is built from $u$. 
	Then, (a) holds true. Also, (b) follows from
	\begin{equation}
	\|v-u\|\leq \|v-\tilde u\|+ \|\tilde u-u\|< \rho+\|\T_\rho u-u\|.
	\end{equation}
	Finally, (c) holds by virtue of $\|\eta-u\|\leq \|\eta-\tilde u\|+\|\tilde u-u\|\leq \rho +\|\T_\rho u-u\|$. 
\end{proof}

\begin{remark}\rm
If $u\in S$ in \eqref{startpoint} is such that $u^H=u$ for all $H\in {\mathcal H}_*$
(which is the case, for instance, if $u^*=u$ and $*$ denotes the usual Schwarz symmetrization
in the space of nonnegative vanishing measurable real functions), then by construction 
$\T_\rho u=u$ for every $\rho>0$ and conclusions 
(b)-(c) of Theorem~\ref{mainthm} improve into 
\begin{equation}
	\label{improvedvv}
\text{$\|v-u\|< \rho$\quad and\quad $\|\eta-u\|\leq \rho$}.	
\end{equation}
Hence, starting with a minimization sequence 
made of symmetric functions yields a new minimization sequence 
satisfying (a)-(e) and full smallness controls (b)-(c) 
of Theorem~\ref{mainthm}. In many concrete cases (although there are some exceptions, 
as pointed out in \cite{vansch}), if a functional does not
increase under polarization, namely condition~\eqref{assumptionpol-00} holds, then it is also 
non-increasing under symmetrization, namely
\begin{equation*}
\text{$f(u^*)\leq f(u),$\,\,\quad for all $u\in S$}.
\end{equation*}
In these cases, starting from an arbitrary minimization sequence $(u_h)\subset S$, first one can
consider the new symmetric minimization sequence $(u_h^*)\subset S$, which already admits a behavior
nicer than that of $(u_h)$, and then apply the variational principle to it, finding a further 
minimization sequence $(v_h)\subset X$ with even nicer additional properties.
\end{remark}

\noindent
In the abstract framework of Definition~\ref{abssym}, using Ekeland's principle in 
complete metric spaces, we can derive the following result. 

\begin{theorem}[{Symmetric Ekeland's principle, I}] 
	\label{ekelcor}
	Let $S\subset X$ be as in Definition \ref{abssym} and let $S'$ be a closed subset of $S$
	satisfying the properties stated in Remark~\ref{restriction}.	
	Assume that $f:S'\to\R\cup\{+\infty\}$ is a proper and lower 
	semi-continuous functional bounded from below such that \eqref{assumptionpol-00} holds (on $S'$).
	Then for all $\rho>0$ and $\sigma>0$ there exists $v\in S'$ such that
	\begin{enumerate}
	\item[(a)] $\|v-v^*\|_V < (2K+1)\rho$; 
	\item[(b)]   $f(w)\geq f(v)-\sigma \|w-v\|,$\quad\text{for all $w\in S'$.}
	\end{enumerate}
	In addition, one can assume that $f(v)\leq f(u)$ and $\|v-u\|\leq \rho+\|\T_\rho u-u\|$, where $u\in S'$
	is some element which satisfies $f(u)\leq \inf f+\sigma\rho$.
\end{theorem}
\begin{proof}
As $S'$ is a closed subset of the Banach space $X$, then $(S',d)$ is a complete
metric space, where $d(u,v)=\|u-v\|$. Given $\rho>0$ and $\sigma>0$, let 
$u\in S'$ with $f(u)\leq \inf f+\sigma\rho$.
If $\T_\rho:S'\to S'$ is the map of Proposition \ref{mapJvS} (cf.\ Remark~\ref{restriction}), let $\tilde u=\T_\rho u\in S'$.
Then $\|\tilde u-u^*\|_V<\rho$ and, taking into account \eqref{assumptionpol-00}, 
$f(\tilde u)\leq \inf f+\sigma\rho$. By applying Ekeland's variational principle 
on the complete metric space $S'$ \cite[Theorem 1.1]{ekeland1}, 
we find $v\in S'$ such that $f(v)\leq f(\tilde u)\leq f(u)$, $\|v-\tilde u\|\leq \rho$ and
$f(w)\geq f(v)-\sigma \|w-v\|$, for every $w\in S'$. As in inequality \eqref{diffsymmcontro}, it readily follows
$\|v-v^*\|_V\leq 2K\|v-\tilde u\|+\|\tilde u-u^*\|_V<(2K+1)\rho$. Finally 
$\|v-u\|\leq \|v-\tilde u\|+\|\T_\rho u-u\|\leq \rho+\|\T_\rho u-u\|$, concluding the proof. 
\end{proof}

\noindent
Notice that, in Banach spaces, essentially conclusion (b) of Theorem \ref{ekelcor} can be recovered by (e) of Theorem~\ref{mainthm}
with $p=1$, since $\|v-\eta\|-\|w-\eta\|\geq -\|w-v\|$ for all $w\in X$.
\vskip1pt

\noindent
On Banach spaces, we can state the following

\begin{theorem}[{Symmetric Ekeland's principle, II}]  
	\label{ekelcorII}
	Assume that $X$ is a Banach space and that $f:X\to\R\cup\{+\infty\}$ is a proper and lower 
	semi-continuous functional bounded from below such that \eqref{assumptionpol-00} holds. Moreover,
	assume that for all $u\in \dom{f}$ there exists $\xi\in S$ such that $f(\xi)\leq f(u)$.
	Then for every $\rho>0$ and $\sigma>0$ there exists $v\in X$ with
	\begin{enumerate}
	\item[(a)] $\|v-v^*\|_V< (K(C_\Theta+1)+1)\rho$; 
	\item[(b)]   $f(w)\geq f(v)-\sigma \|w-v\|,$\quad\text{for all $w\in X$.}
	\end{enumerate}
	In addition, one can assume that $f(v)\leq f(u)$ and $\|v-u\|\leq \rho+\|\T_\rho u-u\|$, where $u\in S$
	is some element which satisfies $f(u)\leq \inf f+\sigma\rho$.
\end{theorem}
\begin{proof}
Let $u\in \dom{f}$ with $f(u)\leq \inf f+\sigma\rho$. Then, let
$\xi\in S$ with $f(\xi)\leq \inf f+\sigma\rho$. At this stage, one can proceed as in the proof
of Theorem \ref{ekelcor}, with Ekeland's principle now applied to $f$ defined on the whole $X$,
yielding a $v\in X$ with the desired properties.
\end{proof}

\noindent
Let now $f,f_h:X\to\R\cup\{+\infty\}$ be lower semi-continuous functionals such that:
\begin{equation}
	\label{gamma1}
\text{for any $u\in\dom{f}\cap S$, there is $(u_h)\subset S$ with $u_h\to u$ and $f_h(u_h)\to f(u)$,}
\end{equation}
and
\begin{equation}
	\label{gamma2}
\liminf_h\big(\inf_X f_h\big)\geq \inf_X f.
\end{equation}
As pointed out in \cite{corvel}, in some sense, this means that 
the function $f$ is the uniform $\Gamma$-limit of the sequence $(f_h)$. 
In the framework of Definition~\ref{abssym} we introduce the following
\begin{definition}
We set $X_{{{\mathcal H}_*}}:=\{u\in S: u^H=u,\,\,\,\text{for all $H\in {\mathcal H}_*$}\}$.
\end{definition}

\begin{remark}\rm
In the framework of Definition~\ref{abssym}, the space $(X_{{{\mathcal H}_*}},\|\cdot\|)$ is complete, as it is closed in $X$. Conversely, assume only that
the conclusion of the symmetric Ekeland principle holds true for the subclass of lower semi-continuous functionals
$f:(X,\|\cdot\|_V)\to \R\cup\{+\infty\}$ bounded from below and
which are not increasing under polarization of elements $u\in S$. Then $(X_{{{\mathcal H}_*}},\|\cdot\|_V)$ is 
complete if $u^H$ is contractive with respect to $\|\cdot\|_V$. In fact, let $(u_h)$ be a Cauchy sequence in $(X_{{{\mathcal H}_*}},\|\cdot\|_V)$. Defining
$f:X\to\R^+$ by $f(u):=\lim_j \|u_j-u\|_V$, then $f$ is continuous and $f(u_h)\to 0$ as $h\to\infty$, 
yielding $\inf f=0$. Observe also that, by contractivity,
\begin{equation*}
f(u^H)=\lim_j \|u_j-u^H\|_V=\lim_j \|u_j^H-u^H\|_V\leq \lim_j \|u_j-u\|_V=f(u), 
\end{equation*}
for all $H\in {\mathcal H}_*$ and $u\in S$ and, for all $u\in X$,
\begin{equation*}
f(\Theta(u))=\lim_j \|u_j-\Theta(u)\|_V=\lim_j \|\Theta(u_j)-\Theta(u)\|_V\leq\lim_j \|u_j-u\|_V=f(u).
\end{equation*}
Given $\eps\in (0,1)$, there is $v\in X$ with $f(v)\leq \eps^2$, $\|v-v^*\|_V<\eps$ and
$f(w)\geq f(v)-\eps\|w-v\|_V$, for all $w\in X$. By choosing $w=u_j$ and letting $j\to\infty$, it holds
$f(v)\leq \eps f(v)$, namely $\|u_h-v\|_V\to 0$ as $h\to\infty$. Moreover $v=v^*$, by the arbitrariness 
of $\eps$. Hence $v\in {\mathcal H}_*$.
\end{remark}

\noindent
Under the above conditions \eqref{gamma1}-\eqref{gamma2}, we have a symmetric version of an Ekeland type principle 
proposed by Corvellec \cite[Proposition 1]{corvel}.

\begin{theorem}[{Symmetric Ekeland's principle, III}]  
	\label{ekelcorIII}
	Assume that $X$ is a Banach space and  that $f,f_h:X\to\R\cup\{+\infty\}$ are proper lower 
	semi-continuous functionals with $f,f_h$ bounded from below satisfying 
	conditions \eqref{gamma1}-\eqref{gamma2}. Moreover, assume that
	\begin{equation}
		\label{assumptionpol-1}
	\text{$f_h(u^H)\leq f_h(u),$\quad for all $u\in S$, $H\in {\mathcal H}_*$ and $h\in\N$}.
	\end{equation}
	Let $Y$ be a nonempty subset of $S$, $\rho>0$ and $\sigma>0$ such that
	$$
	\inf_Y f<\inf_X f+\sigma\rho.
	$$
	Then, for every $h_0\geq 1$ there exist $h\geq h_0$, $m>1$, $(u_h)\subset S$ and $(v_h)\subset X$ such that
	\begin{enumerate}
	\item[(a)] $\|v_h-v^*_h\|_V< (K(C_\Theta+1)+1)\rho$;
	\item[(b)] $|f_h(v_h)-\inf_X f|<\sigma\rho$;
	\item[(c)] $d(v_h,Y)<\rho+\|T_{(m-1)\rho/m}u_h-u_h\|$;
	\item[(d)]   $f_h(w)\geq f_h(v_h)-\sigma \|w-v_h\|,$\quad\text{for all $w\in X$.}
	\end{enumerate}
	In particular, if $f_h=f$ for all $h\in\N$ and $Y\subset X_{{{\mathcal H}_*}}$,
	there exists $v\in X$ such that
	\begin{enumerate}
	\item[(a)] $\|v-v^*\|_V< (K(C_\Theta+1)+1)\rho$;
	\item[(b)] $|f(v)-\inf_X f|<\sigma\rho$;
	\item[(c)] $d(v,Y)<\rho$;
	\item[(d)] $f(w)\geq f(v)-\sigma \|w-v\|,$\quad\text{for all $w\in X$.}
	\end{enumerate}
\end{theorem}
\begin{proof}
Given $h_0\geq 1$, $\rho>0$ and $\sigma>0$, taking into account 
\eqref{gamma1}-\eqref{gamma2}, arguing as in the proof of \cite[Proposition 1]{corvel},
one finds $u\in Y\cap \dom{f}$, $m>1$, $\hat\sigma\in (0,\sigma)$,
$\tilde\sigma\in (\hat\sigma,\sigma)$ with $m\tilde\sigma/(m-1)<\sigma$ 
and $f(u)<\inf f+\hat\sigma\rho$, and points $u_h\in S\cap \dom{f_h}$ such that 
$$
\|u_h-u\|<\rho/m,
\quad\,\,
\inf_X f_h\geq \inf_X f-\frac{(\tilde\sigma-\hat\sigma)\rho}{2},
\quad\,\,
f_h(u_h)\leq f(u)+\frac{(\tilde\sigma-\hat\sigma)\rho}{2},
$$
and, in turn,
\begin{equation}
	\label{startinginequall}
f_h(u_h)<\inf_X f_h+\tilde\sigma\rho.
\end{equation}
By means of~\eqref{assumptionpol-1} condition \eqref{assumptionpol-00} is satisfied for the functionals $f_h$.
Therefore, in light of Theorem~\ref{ekelcorII} (applied to $f_h$, starting from the point $u_h$, see \eqref{startinginequall} above)
with $\sigma$ replaced by $m\tilde\sigma/(m-1)$ and $\rho$ replaced by $(m-1)\rho/m$
respectively, there exist $v_h\in X$ such that 
\begin{align*}
& f_h(v_h)\leq f_h(u_h),\quad 
\|v_h-v^*_h\|_V < (K(C_\Theta+1)+1)\frac{m-1}{m}\rho<(K(C_\Theta+1)+1)\rho,  \\
& f_h(w)\geq f_h(v_h)-\frac{m}{m-1}\tilde\sigma \|w-v_h\|
\geq f_h(v_h)-\sigma \|w-v_h\|,\quad\text{for all $w\in X$.}
\end{align*}
Also, it holds $|f_h(v_h)-\inf f|<\sigma\rho$, since
$$
\inf_X f-\sigma\rho<
\inf_X f-\frac{(\tilde\sigma-\hat\sigma)\rho}{2}
\leq f_h(v_h)\leq f_h(u_h)\leq f(u)+\frac{(\tilde\sigma-\hat\sigma)\rho}{2}<\inf f+\sigma\rho.
$$
Moreover, noticing that $\|v_h-u_h\|<(m-1)\rho/m+\|T_{(m-1)\rho/m}u_h-u_h\|$, it holds
$$
d(v_h,Y)\leq \|v_h-u\|\leq \|v_h-u_h\|+ \|u_h-u\|<\rho+\|T_{(m-1)\rho/m}u_h-u_h\|.
$$
The last conclusion of the statement can be easily obtained by taking into account that
$T_\rho u=u$, for all $\rho>0$ and $u\in Y\subset X_{{{\mathcal H}_*}}$.
\end{proof}

\noindent
Based upon the strong Ekeland's principle stated by Georgiev in \cite{georgiev}, which exhibits
some additional {\em stability features}, we formulate the following symmetric version.

\begin{theorem}[{Symmetric Ekeland's principle, IV}]  
	\label{ekelcor4}
	Assume that $X$ is a Banach space and that $f:X\to\R\cup\{+\infty\}$ is a proper and lower 
	semi-continuous functional bounded from below such that \eqref{assumptionpol-00} holds. 
	Then for every $\rho_1,\rho_2>0$, $\sigma>0$ and $u\in S$ such that
	$$
	f(u)<\inf_X f+\sigma\rho_1,
	$$
	there exists a point $v\in X$ such that
	\begin{enumerate}
	\item[(a)] $\|v-v^*\|_V< (K(C_\Theta+1)+1)(\rho_1+\rho_2)$; 
	\item[(b)]   $f(w)\geq f(v)-\sigma \|w-v\|,$\quad\text{for all $w\in X$.}
	\item[(c)] for every sequence $(u_h)\subset X$, it follows
	$$
	\lim_h (f(u_h)+\sigma \|u_h-v\|)=f(v)\quad\Rightarrow\quad
	\lim_h u_h=v.
	$$
	\end{enumerate}
\end{theorem}
\begin{proof}
Given $\rho_1,\rho_2>0$ and $\sigma>0$, let $u\in S$ be such that $f(u)<\inf f+\sigma\rho_1$. 
If $\T_\rho:S\to S$ is the map of Proposition \ref{mapJvS}, let $\tilde u=\T_{\rho_1+\rho_2} u\in S$.
Then $\|\tilde u-u^*\|_V<\rho_1+\rho_2$ and, taking into account \eqref{assumptionpol-00}, 
$f(\tilde u)<\inf f+\sigma\rho_1$. By \cite[Theorem 1.6]{georgiev} there exists $v\in X$
such that (b) and (c) hold and $\|v-\tilde u\|<\rho_1+\rho_2$. Then 
$\|v-v^*\|_V < (K(C_\Theta+1)+1)(\rho_1+\rho_2)$, by arguing as in the previous proofs.
\end{proof}

\noindent
In some situations, a version of Ekeland's variational principle, sometimes called
{\em altered} principle, has revealed very useful \cite{penot}. 
Here follows a symmetric version of it. A similar statement holds 
with $S$ in place of $X$, when $S$ is closed.

\begin{theorem}[{Symmetric Ekeland's principle, V}]  
	\label{ekelcorV}
	Assume that $X$ is a Banach space and that $f:X\to\R\cup\{+\infty\}$ is a proper and lower 
	semi-continuous bounded below functional such that \eqref{assumptionpol-00} holds. 
	Then, for every $u\in S$, $\rho>0$ and $\sigma>0$ there exists an element $v\in X$ such that
	\begin{enumerate}
	\item[(a)]   $f(w)>f(v)-\sigma \|w-v\|,$\quad\text{for all $w\in X\setminus\{v\}$.}
    \item[(b)]   $f(v)\leq f(u)-\sigma \|v-\T_\rho u\|$.
	\end{enumerate}
	If in addition $u\in X_{{{\mathcal H}_*}}$, then {\rm (b)} strengthens into 
	$f(v)\leq f(u)-\sigma \|v-u\|$.
\end{theorem}
\begin{proof}
Given $u\in S$, $\rho>0$ and $\sigma>0$, consider $\T_\rho u\in S$. 
By applying \cite[Theorem A, p.814]{penot} to $\T_\rho u$ and taking into account that
$f(\T_\rho u)\leq f(u)$ by \eqref{assumptionpol-00}, we get an element $v\in X$ satisfying
properties (a) and (b). 
\end{proof}

\begin{remark}\rm
	\label{ekelcorV-rmk}
	Let $u\in S$ be such that $f(u)\leq \inf f+\rho\sigma$, for 
	some $\rho,\sigma>0$. Then, in addition to the conclusions of Theorem~\ref{ekelcorV},
	it follows $\|v-v^*\|_V\leq \rho$, as in the previous statements. In fact, in light of (b) of Theorem~\ref{ekelcorV}, we have
	$$
	\|v-\T_\rho u\|\leq \frac{f(u)-f(v)}{\sigma}\leq \frac{f(u)-\inf f}{\sigma}\leq\rho,
	$$
	which is turn allows to get the desired conclusion taking into account that
	$\|\T_\rho u-u^*\|_V<\rho$. Also, one has $\|v-u\|\leq \rho+\|\T_\rho u-u\|$.
	In other words, Theorem~\ref{ekelcorV} is stronger than the previous
	statements in the fact that it holds for {\em any} point $u\in S$. On the other hand, the price 
	to be paid is that the {\em location} of $v$ with respect to $u$ is {\em no longer available}
	and it is recovered provided that $f(u)\leq \inf f+\rho\sigma$.
\end{remark}

\noindent
Let $X'$ denote the topological dual space of $X$.
We need to recall from \cite{deville} the following 

\begin{definition}
Let $X$ be a Banach space, $\beta$ a family of bounded 
subsets of $X$ which constitutes a bornology, $f:X\to\R\cup\{+\infty\}$
a functional and $u\in\dom{f}$. We say that $f$ is $\beta$-differentiable at $u$
with $\beta$-derivative $\varphi=f'(u)\in X'$ if
$$
\lim_{t\to 0}\frac{f(u+tw)-f(u)-\langle\varphi,tw \rangle}{t}=0
$$
uniformly for $w$ inside the elements of $\beta$. We denote by $\tau_\beta$ the topology on $X'$
of uniform convergence on the elements of $\beta$. 
\end{definition}

\noindent
When $\beta$ is the class of {\rm all} bounded subsets of $X$,
then the $\beta$-differentiability coincides with Fr\'echet differentiability and $\tau_\beta$ coincides
with the norm topology on $X'$. When $\beta$ is the class of all {\rm singletons} of $X$,
the $\beta$-differentiability coincides with Gateaux differentiability and $\tau_\beta$ is
the weak* topology on $X'$.
\vskip3pt
\noindent
We consider the Banach space $(X_\beta,\|\cdot\|_\beta)$ defined as follows
\begin{align*}
X_\beta &:=\{g:X\to\R:\, \text{$g$ is bounded, Lipschitzian and $\beta$-differentiable on $X$}\}, \\
\|g\|_\beta &:=\|g\|_\infty+\|g'\|_\infty,\qquad        \|g\|_\infty=\sup_{u\in X} |g(u)|,\,\,\,\,
\|g'\|_\infty=\sup_{u\in X} \|g'(u)\|.
\end{align*}

\begin{definition}
We say that $b\in X_\beta$ is a bump function if ${\rm supt}(b)\neq\emptyset$ is bounded.
\end{definition}

\noindent
Next we recall a localized version of Deville-Godefroy-Zizler's 
variational principle (see \cite[Corollary II.4 and Remark II.5]{deville}).

\begin{theorem}[{Deville-Godefroy-Zizler's principle}]
	\label{deville}
Assume that $X$ is a Banach space which admits a bump function in $X_\beta$ and let
$f:X\to\R\cup\{+\infty\}$ be a proper and lower 
semi-continuous functional bounded from below. 
Then there exists a positive number ${\mathcal A}$ such that, for all $\eps\in (0,1)$,
and $u\in X$ with $f(u)<\inf f+{\mathcal A}\eps^2$, there exists $g\in X_\beta$ and $v\in X$ such that
\begin{enumerate}
\item[(a)]  $\|v-u\|\leq \eps$;
\item[(b)]   $\|g\|_\infty\leq \eps$ and $\|g'\|_\infty\leq \eps$;
\item[(c)]   $f(w)+g(w)\geq f(v)+g(v),$\quad\text{for all $w\in X$}.
\end{enumerate}
\end{theorem}

\noindent
Next, we state a symmetric version of Deville-Godefroy-Zizler's variational principle.

\begin{theorem}[{Symmetric Deville-Godefroy-Zizler's principle}]
	\label{mainthm-deville}
Assume that $X$ is a Banach space which admits a bump function in $X_\beta$ and let
$f:X\to\R\cup\{+\infty\}$ be a proper and lower 
semi-continuous functional bounded from below satisfying~\eqref{assumptionpol-00}.
Then there exists a positive number ${\mathcal A}$ such that, for all $\eps\in (0,1)$,
and $u\in S$ with $f(u)<\inf f+{\mathcal A}\eps^2$, there exists $g\in X_\beta$ and $v\in X$ such that
\begin{enumerate}
\item[(a)] $\|v-v^*\|_V< (K(C_\Theta+1)+1)\eps$; 
\item[(b)] $\|v-u\|\leq \eps+\|\T_\eps u-u\|$;
\item[(c)]  $\|g\|_\infty\leq \eps$ and $\|g'\|_\infty\leq \eps$;
\item[(d)]   $f(w)+g(w)\geq f(v)+g(v),$\quad\text{for all $w\in X$.}
\end{enumerate}
\end{theorem}
\begin{proof}
	By Theorem~\ref{deville}, there exists a positive number ${\mathcal A}$ with the stated properties.	
	Let $u\in S$ and $\eps\in (0,1)$ such that $f(u)<\inf f+{\mathcal A}\eps^2$.
	If $\T_\eps:S\to S$ is as in Proposition~\ref{mapJvS},
	we set $\tilde u:=\T_\eps u\in S$. By construction we have $\|\tilde u-u^*\|_V<\eps$
	and $f(\tilde u)<\inf f+{\mathcal A}\eps^2$. Hence, by the just stated principle,
	there is $g\in X_\beta$ with $\|g\|_\infty\leq \eps$ and $\|g'\|_\infty\leq \eps$
	and $v\in X$ such that $\|v-\tilde u\|\leq \eps$ and $f(w)+g(w)\geq f(v)+g(v)$, for every $w\in X$.
	Furthermore, we have $\|v-v^*\|_V< (K(C_\Theta+1)+1)\eps$ with the usual argument, as well as
	$\|v-u\|\leq \|v-\tilde u\|+\|\tilde u-u\|\leq \eps+\|\T_\eps u-u\|$. 
	This concludes the proof.
\end{proof}

\subsection{Statements with weights}

\noindent
In this section we will derive a symmetric version of Ekeland's variational
principle with weights (see also \cite{ekelandbook,suzuki,zhong}), based upon the following result due to Zhong (take $x_0=y$
in \cite[Theorem 1.1]{zhong}). The result is often used to prove that
a lower semi-continuous bounded below functional which satisfies a
suitable weighted Palais-Smale condition needs to be coercive.

\begin{theorem}[Zhong's principle]
	\label{zhongth}
Let $X$ be a complete metric space and consider a nondecreasing and continuous function
$h:[0,+\infty)\to [0,+\infty)$ such that
$$
\int_0^{+\infty}\frac{1}{1+h(s)}ds=+\infty.
$$
Assume that $f:X\to\R\cup\{+\infty\}$ is a proper lower 
semi-continuous functional bounded from below. Let $u\in X$, $\rho>0$ and $\sigma>0$ such that
$$
f(u)<\inf_X f+\sigma\rho.
$$ 
Then there exists $v\in X$ such that
\begin{enumerate}
\item[(a)]   $f(v)\leq f(u)$;
\item[(b)]   $d(v,u)\leq r(\rho)$;
\item[(c)]   $f(w)\geq f(v)-\sigma\frac{d(w,v)}{1+h(d(v,u))}$\quad\text{for all $w\in X$,}
\end{enumerate}
where $r(\rho)$ is a positive number which satisfies 
$$
\int_0^{r(\rho)}\frac{1}{1+h(s)}ds\geq \rho.
$$
\end{theorem}

\noindent
As a consequence, in the framework of Definition~\ref{abssym},
we obtain the following

\begin{theorem}[Symmetric Zhong's principle]
	\label{symzhong}
Let $X$ be a Banach space and consider a nondecreasing continuous function
$h:[0,+\infty)\to [0,+\infty)$ such that 
$$
\int_0^{+\infty}\frac{1}{1+h(s)}ds=+\infty.
$$
Assume that for $\rho_0>0$ sufficiently small, 
there exists a function $r:[0,\rho_0)\to [0,\infty)$ with
\begin{equation}
	\label{rhocondition}
\int_0^{r(\rho)}\frac{1}{1+h(s)}ds\geq \rho, \,\,\qquad \lim_{\rho\to 0^+} r(\rho)=0.
\end{equation}
Let $f:X\to\R\cup\{+\infty\}$ be a proper lower semi-continuous functional bounded from below
such that condition \eqref{assumptionpol-00} holds. Let $u\in S$, $\rho>0$ and $\sigma>0$ be such that
\begin{equation}
	\label{startingpp}
f(u)<\inf_X f+\sigma\rho
\end{equation}
Then there exists $v\in X$ such that
\begin{enumerate}
\item[(a)]   $\|v-v^*\|_V< (K(C_\Theta+1)+1)r(\rho)$;
\item[(b)]   $f(v)\leq f(u)$;
\item[(c)]   $\|v-u\|\leq r(\rho)+\|\T_{r(\rho)}u-u\|$;
\item[(d)]   $f(w)\geq f(v)-\sigma\frac{\|w-v\|}{1+h(\|v-\T_{r(\rho)}u\|)}$\quad\text{for every $w\in X$.}
\end{enumerate}
\end{theorem}
\begin{proof}
Let $u\in S$, $\rho>0$ and $\sigma>0$ with $f(u)< \inf f+\sigma\rho$. Let also $r(\rho)$ be a positive number
which satisfied conditions~\eqref{rhocondition}. Then, if $\T_{r(\rho)}:S\to S$ is the map of 
Proposition \ref{mapJvS}, let $\tilde u:=\T_{r(\rho)} u\in S$.
Then $\|\tilde u-u^*\|_V<r(\rho)$ and, taking into account \eqref{assumptionpol-00}, we can conclude
$f(\tilde u)< \inf f+\sigma\rho$. By applying Theorem~\ref{zhongth} to this element $\tilde u$, we find an element
$v\in X$ such that $\|v-\tilde u\|\leq r(\rho)$, $f(v)\leq f(\tilde u)\leq f(u)$ and
$$
f(w)\geq f(v)-\sigma\frac{\|w-v\|}{1+h(\|v-\T_{r(\rho)} u\|)},\quad\,\,\,\text{for every $w\in X$}. 
$$
Also, we have
$$
\|v-u\|\leq \|v-\tilde u\|+\|\T_{r(\rho)}-u\|\leq r(\rho)+\|\T_{r(\rho)}-u\|.
$$
We conclude with $\|v-v^*\|_V\leq K(C_\Theta+1)\|v-\tilde u\|
+\|\tilde u-u^*\|_V<(K(C_\Theta+1)+1)r(\rho)$.
\end{proof}

\begin{remark}\rm
In the case $h\equiv 0$, one finds precisely the symmetric version of the classical
Ekeland's variational principle (notice that one can take $r(\rho)=\rho$). In the Cerami
case $h(s)=s$ \cite{cerami}, one can take $r(\rho)=e^\rho-1$ and the conclusion of Theorem~\ref{symzhong}
reads as: for every $u\in S$ which satisfies~\eqref{startingpp} with $\rho>0$ and $\sigma>0$ 
there exists $v\in X$ such that
\begin{enumerate}
\item[(a)]   $\|v-v^*\|_V < (K(C_\Theta+1)+1)(e^\rho-1)$;
\item[(b)]   $f(v)\leq f(u)$;
\item[(c)]   $\|v-u\|\leq e^\rho-1+\|\T_{e^\rho-1}u-u\|$;
\item[(d)]   $f(w)\geq f(v)-\sigma\frac{\|w-v\|}{1+\|v-\T_{e^\rho-1}u\|}$,\quad\text{for all $w\in X$.}
\end{enumerate}
Furthermore, if $u\in X_{{{\mathcal H}_*}}$ and  $\rho=\sigma>0$, then 
there exists $v\in X$ such that
\begin{enumerate}
\item[(a)]   $\|v-v^*\|_V < (K(C_\Theta+1)+1)(e^\rho-1)$;
\item[(b)]   $f(v)\leq f(u)$;
\item[(c)]   $\|v-u\|\leq e^\rho-1$;
\item[(d)]   $f(w)\geq f(v)-\rho\frac{\|w-v\|}{1+\|v-u\|}$,\quad\text{for all $w\in X$.}
\end{enumerate}
\end{remark}

\noindent
Next, we will highlight some by-products of the previous principles
in the context of non-smooth critical point theory.
We recall the definition of weak slope \cite{dm}. $B(u,\delta)$ stands
for the open ball in $X$ of center $u$ and radius $\delta$ and 
${\rm epi}(f)=\{(u,\lambda)\in X\times\R: f(u)\leq\lambda\}$.

\begin{definition}\label{defslope}
For every $u\in X$ with $f(u)\in\R$, we denote by $|df|(u)$ the supremum
of $\sigma$'s in $[0,\infty)$ such that there exist $\delta>0$ and a continuous map
$$
{\mathcal H}:\,B((u,f(u)), \delta)\cap {\rm epi}(f) \times[ 0, \delta]  \to X,
$$
satisfying, for all $(\xi,\mu)\in B((u,f(u)), \delta)\cap {\rm epi}(f)$
and $t\in [0,\delta]$,
\begin{equation*}
\|{\mathcal H}((\xi,\mu),t)-\xi\| \leq t,\qquad
f({\mathcal H}((\xi,\mu),t)) \leq f(\xi)-\sigma t.
\end{equation*}
The extended real number $|df|(u)$ is called the weak slope of $f$ at $u$.
\end{definition}

\begin{remark}\rm
	\label{weakstrong}
If $f$ is of class $C^1$, then $|df|(u)=\|df(u)\|$, see~\cite[Corollary 2.12]{dm}. If $u\in X$ with $f(u)<+\infty$
the strong slope of $f$ at $u$ \cite{degiorgimarinotosq} is the extended real $|\nabla f|(u)$,
$$
|\nabla f|(u):=
\begin{cases}
\displaystyle\limsup_{\xi \to u}\frac{f(u)-f(\xi)}{d(u,\xi)} & \text{if $u$ is not a local minimum for $f$;}  \\	
0  & \text{if $u$ is a local minimum for $f$}.
\end{cases}
$$
It easily follows from the definition that $|df|(u)\leq |\nabla f|(u)$.
\end{remark}

\noindent
We can now state the following 

\begin{corollary}
Let $X$ be a Banach space and $h:[0,+\infty)\to [0,+\infty)$ a nondecreasing and continuous function 
such that 
$$
\int_0^{+\infty}\frac{1}{1+h(s)}ds=+\infty.
$$
Assume that for $\rho_0>0$ sufficiently small, 
there exists a function $r:[0,\rho_0)\to [0,\infty)$ with
$$
\int_0^{r(\rho)}\frac{1}{1+h(s)}ds\geq \rho, \,\,\qquad \lim_{\rho\to 0^+} r(\rho)=0.
$$
Let $f:X\to\R\cup\{+\infty\}$ be a proper lower semi-continuous functional bounded from below
such that \eqref{assumptionpol-00} holds. Then for every $\rho>0$ and $u_\rho\in S$ with 
$$
f(u_\rho)<\inf_X f+\rho^2
$$ 
there exists $v_\rho\in X$ such that
\begin{enumerate}
\item[(a)]   $\|v_\rho-v^*_\rho\|_V < (K(C_\Theta+1)+1)r(\rho)$;
\item[(b)]   $(1+h(\|v_\rho-\T_{r(\rho)}u_\rho\|))|df|(v_\rho)\leq\rho,$\quad\text{for all $w\in X$.}
\end{enumerate}
In particular, for every minimizing sequence $(u_j)\subset S$ for $f$, there exists a 
minimizing sequence $(v_j)\subset X$ and $(\mu_j)\subset\R^+$ with $\mu_j\to 0$ such that
$$
\lim_{j\to\infty}\|v_j-v^*_j\|_V=0,\qquad
\lim_{j\to\infty}(1+h(\|v_j-\T_{\mu_j}u_j\|))|df|(v_j)=0.
$$
Moreover, for every symmetric minimizing sequence $(u_j)\subset X_{{{\mathcal H}_*}}$ for $f$,
there exists a minimizing sequence $(v_j)\subset X$ such that
$$
\lim_{j\to\infty}\|v_j-v^*_j\|_V=0,\qquad
\lim_{j\to\infty}(1+h(\|v_j-u_j\|))|df|(v_j)=0.
$$
\end{corollary}
\begin{proof}
Taking into account Remark~\ref{weakstrong}, it is an easy consequence of Theorem~\ref{symzhong}.
\end{proof}

\subsection{Statements with constraints} 
A symmetric version of Ekeland's principle {\em with constraints}, 
in the spirit of \cite[Theorem 3.1]{ekeland1}, can also be formulated. 
Assume that $G_j:X\to\R$ with $1\leq j\leq m$ are $C^1$ functions, 
let $1\leq p\leq m$ and consider the set
\begin{equation*}
{\mathscr C}=\{u\in X: \text{$G_j(u)=0$ for $1\leq j\leq p$ and $G_j(u)\geq 0$ for $p+1\leq j\leq m$}\}.
\end{equation*}
For all $u$ in ${\mathscr C}$, we denote by ${\mathscr I}(u)$ the index set of saturated constraints (cf.~\cite{ekeland1}), namely
$j\in {\mathscr I}(u)$ if and only if $G_j(u)=0$. We consider the following assumptions.
\begin{equation}
	\label{basicregularr}
\text{$f:X\to\R$ is Fr\'echet differentiable,
\,\,\qquad  $-\infty<\inf\limits_{{\mathscr C}}f<+\infty$;}
\end{equation}
\begin{equation}
	\label{inssSS}
\text{for all $u\in {\mathscr C}$ there exists $\xi\in {\mathscr C}\cap S$ such that $f(\xi)\leq f(u)$;}
\end{equation}
\begin{equation}
		\label{linearindip}
\text{for all $u\in {\mathscr C},$ the elements $\{dG_j(u)\}_{j\in {\mathscr I}(u)}$ are linearly independent in $X'$;} 
\end{equation}
\begin{equation}
	\label{simmfunctconstra}
\begin{cases}
\forall u\in {\mathscr C}\cap S,\,\,\,\forall H\in {\mathcal H}_*:\quad u^H\in {\mathscr C}, &\\
\noalign{\vskip2.1pt}
\forall u\in {\mathscr C}\cap S,\,\,\,\forall H\in {\mathcal H}_*:\quad f(u^H)\leq f(u).
\end{cases}
\end{equation}
Then, for every $\eps>0$, there exists $u_\eps\in {\mathscr C}$ such that
$$
f(u_\eps)\leq \inf_{{\mathscr C}}f+\eps^2,
\quad
\|u_\eps-u_\eps^*\|_V <C\eps,
\quad
\Big\|df(u_\eps)-\sum_{j=1}^m\lambda_j dG_j(u_\eps)\Big\|_{X'}\leq \eps,
$$	
for some $\lambda_j\in\R$, $1\leq j\leq m$, such that $\lambda_j\geq 0$ 
for $p+1\leq j\leq m$ and $\lambda_j=0$ if $G_j(u_\eps)\neq 0$. 
The assertion follows by applying Theorem~\ref{ekelcorII} 
to the functional $\hat f:X\to\R\cup\{+\infty\}$
\begin{equation*}
\hat f(u):=
\begin{cases}
f(u) & \text{for $u\in {\mathscr C}$} \\
+\infty & \text{for $u\in X\setminus{\mathscr C}$},
\end{cases}
\end{equation*}
finding almost symmetric point $u_\eps\in {\mathscr C}$ such that
$f(u_\eps)\leq \inf f|_{{\mathscr C}}+\eps^2$ and
$$
\forall w\in {\mathscr C}:\quad\, f(w)\geq f(u_\eps)-\eps\|w-u_\eps\|,
$$
and then arguing exactly as in the proof of \cite[Theorem 3.1]{ekeland1}, namely using
\cite[Lemmas 3.2 and 3.3]{ekeland1}, in view of assumptions~\eqref{basicregularr}-\eqref{linearindip}.
The assumptions of Theorem~\ref{ekelcorII} are fulfilled since $\hat f$ is lower semi-continuous, bounded from below being $f$
bounded from below on ${\mathscr C}$ and, in light of \eqref{simmfunctconstra}, 
it satisfies $\hat f(u^H)\leq \hat f(u)$, for every $u\in S$ and $H\in {\mathcal H}_*$. 
Also, by virtue of \eqref{inssSS}, for all $u\in \dom{\hat f}$ there exists $\xi\in S$ such that $\hat f(\xi)\leq \hat f(u)$.
In the case of a single constraint, namely $m=1$,  then assumption~\eqref{linearindip} 
reads as: $G(u)=0$ implies $dG(u)\neq 0$. On the concrete side,
\eqref{simmfunctconstra} is satisfied in various situations, meaningful in the calculus of variations, such as
$G:W^{1,p}(\R^N)\to\R$, $G(u)=\int_{\R^N}H(|u|)-1$ for suitable $H\in C^1(\R)$
and functionals $f:W^{1,p}(\R^N)\to\R\cup\{+\infty\}$ discussed in Section~\ref{sectionPDEs}.

\subsection{Symmetry, coercivity and PS conditions}
	\label{remsymcoercps}
A sequence $(u_h)\subset X$ is said to be a Palais-Smale ($(PS)$, in short) sequence for $f\in C^1(X)$ if 
$(f(u_h))$  is bounded and $\|df(u_h)\|_{X'}\to 0$, as $h\to\infty$. Also, we say that 
$f$ satisfies the $(PS)$ condition, if each $(PS)$ sequence admits a converging subsequence.
If $f$ is bounded from below and satisfies the $(PS)$ condition, then it is 
coercive \cite{cakliwill,corvel}, meaning that 
$$
\liminf_{\|u\|\to +\infty} f(u)=+\infty.
$$
Actually, an even more general property holds and it is sufficient to assume that the $(PSB)$ condition holds, namely each 
$(u_h)\subset X$ with $(f(u_h))$ bounded and $\|df(u_h)\|_{X'}\to 0$, 
is bounded (see \cite[Corollary 1]{corvel}, for details). 
As pointed out in \cite[Section 10]{mawhwill}, a typical argument to prove the above conclusion is based upon
a clever application of Ekeland's principle, after observing that a violation of the coercivity
yields $\ell\in\R$, $\ell\geq\inf f$ and a sequence $(u_h)\subset X$ such that $f(u_h)\leq \ell+\gamma_h$ and $\|u_h\|\geq h$,
where $(\gamma_h)\subset\R^+$ is a given sequence with $\gamma_h\to 0$ as $h\to\infty$.
Notice that $\ell+\gamma_h-\inf f>0$, for all $h\in\N$. 
Let $\sigma_h>0$ with $\sigma_h\to 0$ as $h\to\infty$ and
$\rho_h=\frac{\ell+\gamma_h-\inf f}{\sigma_h}>0$ be such that $\rho_h\leq h/2$, yielding
\begin{equation*}
f(u_h)\leq \inf f+\sigma_h\rho_h,\quad\, h\in\N.
\end{equation*}
Under reasonable assumptions, we can also have 
$(u_h)\subset S$. At this stage, if \eqref{assumptionpol-00} holds,
\begin{equation*}
f(\T_{\rho_h}u_h)\leq f(u_h)\leq \inf f+\sigma_h\rho_h,\qquad \|\T_{\rho_h}u_h-u_h^*\|_V<\rho_h. 
\end{equation*}
Then, Ekeland's principle yields $(v_h)\subset X$ with $f(v_h)\leq f(u_h)\leq \ell+\gamma_h,$ 
$\|df(v_h)\|_{X'}\leq \sigma_h$ and $\|v_h-\T_{\rho_h}u_h\|\leq\rho_h$, implying $\|v_h-v_h^*\|_V<C\rho_h$.
Notice that, assuming $\|u^H\|=\|u\|$ for all $u\in S$ and $H\in {\mathcal H}_*$, 
which is reasonable for applications to PDEs, there holds
$\|v_h\|\geq \|\T_{\rho_h}u_h\|-\rho_h=\|u_h\|-\rho_h\geq h/2$,
yielding $\|v_h\|\to\infty$, as $h\to\infty$. In particular, it follows $f(v_h)\to\ell$, since
$$
\ell=\liminf_{\|u\|\to +\infty} f(u)\leq \liminf_h f(v_h)\leq \limsup_h f(v_h)\leq\lim_h \,(\ell+\gamma_h)=\ell.
$$
Since $\sigma_h\to 0$, $(v_h)$
in an unbounded Palais-Smale sequence, contradicting the $(PSB)$ condition.  
To guarantee that, in addition, $\|v_h-v_h^*\|_V\to 0$, one would need that $\rho_h\to 0$. On the other hand   
$\rho_h,\sigma_h,\gamma_h\to 0$, by $\sigma_h\rho_h=\ell+\gamma_h-\inf f$,
yields $\ell=\inf f$, which is not the case, in general. 
In conclusion, this argument does not seem to allow
obtaining a true unbounded {\em almost symmetric Palais-Smale sequence}, which would of course considerably improve the statement 
on coercivity, replacing $(PSB)$ with some symmetric version of it
involving Palais-Smale sequences $(u_h)$ with $\|u_h-u_h^*\|_V\to 0$. 
If $f$ is bounded from below, \eqref{assumptionpol-00} holds, and it satisfies 
the symmetric $(PSB)$ condition, then
$$
\liminf_{\|u\|\to +\infty} f(u)>\inf_X f.
$$
It is sufficient to argue by contradiction and let $\ell=\inf f$ in the previous proof, allowing
$\rho_h,\sigma_h,\gamma_h\to 0$. The relationships between 
the symmetry of the functional, its coercivity and Palais-Smale 
conditions of some kind would deserve further attention.

\subsection{Symmetric quasi-convex PS sequences}

To the author's knowledge, the next notion was firstly introduced by Bartsch and 
Degiovanni in \cite{bartschdeg}.

\begin{definition}
Let $X$ be a Banach space, let $f:X\to\R$ be a lower semi-continuous functional 
and $u\in X$. We define the functional $Q_u:X\to\bar\R$ by setting
$$
Q_u(w):=\limsup_{\substack{z\to u \\ \zeta\to w \\ t\to 0}}\frac{f(z+t\zeta)+f(z-t\zeta)-2f(z)}{t^2},
\quad\,\,\text{for every $w\in X$}.
$$
\end{definition}

\noindent
In the framework of Definition~\ref{abssym}, we also introduce the following

\begin{definition}
Let $X$ be a Banach space and let $f:X\to\R$ be a lower semi-continuous functional.
We say that $(u_h)\subset X$ is a {\em symmetric quasi-convex Palais-Smale sequence} 
at level $c\in\R$ ($(SQPS)_c$-sequence, in short) if 
\begin{equation*}
\lim_{h\to\infty}f(u_h)=c,\qquad
\lim_{h\to\infty}|df|(u_h)=0,
\end{equation*}
and, in addition,
\begin{equation}
	\label{addproperty}
\lim_{h\to\infty}\|u_h-u^*_h\|_V=0,\qquad
\liminf_{h\to\infty} Q_{u_h}(w)\geq 0,\,\,\,\,\forall w\in X.
\end{equation}
We say that $f$ satisfies the {\em symmetric 
quasi-convex Palais-Smale condition} at level $c$, $(SQPS)_c$,
in short, if every $(SQPS)_c$-sequence which admits a subsequence 
strongly converging in $W$, up to a subsequence, converges strongly in $X$.
\end{definition}

\noindent
Compared to a standard Palais-Smale sequence, two additional information are involved
on $(u_h)$, a quasi-symmetry and a quasi-convexity condition.

\begin{remark}\rm
As pointed out in \cite{lionscmp,fanghou}, the fact that a Palais-Smale sequence $(u_h)\subset X$
for a functional $f:X\to\R$ of class $C^2$ satisfies the additional second order condition
\begin{equation*}
\liminf_{h\to\infty} \langle f''(u_h)w,w\rangle\geq 0,\quad\text{for all $w\in X$},
\end{equation*}
can sometimes be crucial for the proof of the strong convergence of $(u_h)$ itself to some limit point
$u\in X$. Furthermore, the additional symmetry condition $\|u_h-u^*_h\|_V\to 0$, as $h\to\infty$,
usually provides compactifying effects (see e.g.~\cite[Section 4.2]{vansch}). 
Based upon these considerations, it is quite clear that, in some sense, 
the $(SQPS)_c$-condition is much weaker than the standard Palais-Smale condition. 
Of course $Q_u(w)=\langle f''(u)w,w\rangle$ when $f$ is of class $C^2$ and replacing 
$\langle f''(u)w,w\rangle$ with $Q_u(w)$ appears to be a natural extension 
when the function is not $C^2$ smooth.
\end{remark}

\vskip2pt
\noindent
Let now $X$ be a Hilbert space and consider the following assumptions:
\begin{align}
	\label{assumptionpol}
& \text{$f(u^H)\leq f(u)$ for all $u\in S$ and $H\in {\mathcal H}_*$};  \\
	\label{seq1}
& \text{for all $X$ there exists $\xi\in S$ such that $f(\xi)\leq f(u)$}; \\
\label{seq2}
& \text{if $(u_h)\subset X$ is bounded, then $(\xi_h)\subset S$ is bounded}; \\
	\label{seq3}
& \text{$\|u^H\|\leq \|u\|$ for all $u\in S$ and $H\in {\mathcal H}_*$}; \\
	\label{seq4}
& \text{$f$ admits a bounded minimizing sequence}.
\end{align}

\noindent
Notice that assumptions \eqref{seq1}-\eqref{seq3} are satisfied in many typical concrete 
situations, like when $X$ is a Sobolev space $W^{1,p}_0(\Omega)$, $\Omega$ a ball or $\R^N$, $S$ is the cone of its positive
functions and the functional satisfies $f(|u|)\leq f(u)$, for all $u\in X$. Assumption~\eqref{seq4}
is mild but not automatically satisfied of course; for instance all the minimizing sequences for the exponential
function on $\R$ are unbounded.
\noindent
We can now state the following

\begin{theorem}
		\label{maincor}
Assume that $f:X\to\R$ is a lower semi-continuous functional bounded from below such 
that conditions \eqref{assumptionpol}-\eqref{seq4} hold. Then $f$ admits a $(SQPS)_{\inf f}$-sequence.
\end{theorem}
\begin{proof}
	In the course of the proof $C$ will denote a generic 
	constant that might change from line to line. By means of assumption~\eqref{seq4}, we can find 
	a bounded minimizing sequence $(u_h)\subset X$ for $f$, namely
	there exists a sequence $(\eps_h)\subset\R^+$, with $\eps_h\to 0$ as $h\to\infty$, such that
	$\|u_h\|\leq C$ and $f(u_h)<\inf f+\eps_h^3$, for all $h\in\N$. 
	In light of assumptions \eqref{seq1}-\eqref{seq2}, there exists a sequence $(\xi_h)\subset S$ such that 
	$\|\xi_h\|\leq C$ and $f(\xi_h)<\inf f+\eps_h^3$, for all $h\in\N$. Taking into account that any norm $\|\cdot\|$ on $X$
	is a Kadec norm and that assumption~\eqref{assumptionpol} holds,
	by Theorem~\ref{mainthm} (symmetric Borwein-Preiss's principle) with $p=2$, $\sigma_h=\rho_h=\eps_h$, we find two sequences
	$(v_h)\subset X$ and $(\eta_h)\subset X$ such that $\|v_h-v_h^*\|_V<C\eps_h$, $f(v_h)<\inf f+\eps_h^3$ as well as
	\begin{align}
		\label{2controllT}
&	\|v_h-\xi_h\|< \eps_h+\|\T_{\eps_h}\xi_h-\xi_h\|,\quad
	\|\eta_h-\xi_h\|\leq\eps_h+\|\T_{\eps_h}\xi_h-\xi_h\|, \\
		\label{varrel}
&	f(w)\geq f(v_h)+\eps_h (\|v_h-\eta_h\|^2-\|w-\eta_h\|^2),\quad\text{for all $w\in X$.}
	\end{align}
	Fixed any $\zeta\in X$ and $t\in\R$, substituting $w:=v_h+t\zeta$ and $w:=v_h-t\zeta$ into~\eqref{varrel} yields
	\begin{align*}
	f(v_h+t\zeta)&\geq f(v_h)+\eps_h (\|v_h-\eta_h\|^2-\|v_h-\eta_h+t\zeta\|^2),   \\
	f(v_h-t\zeta)&\geq f(v_h)+\eps_h (\|v_h-\eta_h\|^2-\|v_h-\eta_h-t\zeta\|^2). 
	\end{align*}
	Whence, taking into account the parallelogram law, it holds
	\begin{equation}
	f(v_h+t\zeta)+f(v_h-t\zeta)-2f(v_h)\geq -2\eps_h t^2\|\zeta\|^2,\quad\text{for all $\zeta\in X$ and $t\in\R$.}
	\end{equation}
	In turn, for every $w\in X$, it holds
	\begin{align*}
	Q_{v_h}(w) & =\limsup_{\substack{z\to v_h \\ \zeta \to w \\ t\to 0}}\frac{f(z+t\zeta)+f(z-t\zeta)-2f(z)}{t^2} \\
	&\geq \limsup_{\substack{\zeta \to w \\ t\to 0}}\frac{f(v_h+t\zeta)+f(v_h-t\zeta)-2f(v_h)}{t^2} \\
	&\geq -2\eps_h \|w\|^2,
	\end{align*}
	which yields the desired property on $Q_{v_h}$. Notice also that, from \eqref{varrel}, for every $h$ and $w\neq v_h$
	\begin{align*}
	\frac{f(v_h)-f(w)}{\|w-v_h\|}\leq \eps_h\frac{\|w-\eta_h\|^2-\|v_h-\eta_h\|^2}{\|w-v_h\|}\leq 
	\eps_h(\|w-\eta_h\|+\|v_h-\eta_h\|).
	\end{align*}
	By repeatedly applying~\eqref{seq3},
	we get $\|\T_{\eps_h}\xi_h\|\leq \|\xi_h\|$. Whence, recalling \eqref{2controllT}, it follows that
	\begin{align*}
	|df|(v_h) & \leq |\nabla f|(v_h)=\limsup_{w\to v_h}\frac{f(v_h)-f(w)}{\|w-v_h\|}\leq 2\eps_h \|v_h-\eta_h\| \\
	& \leq 2\eps_h \|v_h-\xi_h\|+2\eps_h \|\xi_h-\eta_h\| \\
	\noalign{\vskip3pt}
	& < 4\eps_h (\eps_h+\|\T_{\eps_h}\xi_h-\xi_h\|) 
	\leq 4\eps_h^2+8\eps_h \|\xi_h\|\leq C\eps_h.
	\end{align*}
	This concludes the proof.
\end{proof}

\noindent
In the framework of Definition~\ref{abssym} we also introduce the following
\begin{definition}
	\label{symmebeddddd}
We set $X_*:=\{u\in S: u^*=u\}$ and we say that 
$X$ is symmetrically embedded into $W$ if $\|u^*\|\leq \|u\|$ for all $u\in X$ and
the injection $i:X_*\hookrightarrow W$ is compact.
\end{definition}

\noindent
As a consequence of Theorem~\ref{maincor}, we have the following

\begin{corollary}
	\label{corsym11}
Let $X$ be symmetrically embedded in $W$ and let 
$f:X\to\R$ be a lower semi-continuous functional bounded from below such that \eqref{assumptionpol}-\eqref{seq4} hold.	
Then $f$ admits a $(SQPS)_{\inf f}$-sequence converging weakly in $X$ and strongly in $W$. 
If in addition $(SQPS)_{\inf f}$ holds, there exists a point 
$z\in S$ such that $f(z)=\inf f$, $|df|(z)=0$, $z=z^*$ and $Q_z\geq 0$.
\end{corollary}
\begin{proof}
	Let $C$ denote a generic constant that might change from line to line.
	By Theorem~\ref{maincor}, $f$ admits a 
	$(SQPS)_{\inf f}$-sequence $(v_h)\subset X$. By construction $(v_h)$ is bounded
	in $X$. In fact, with the notations in the proof of Theorem~\ref{maincor}, 
	there exist a vanishing sequence $(\eps_h)\subset\R^+$ 
	and a bounded sequence $(\xi_h)\subset S$, yielding
	$$
	\|v_h\|\leq \|v_h-\xi_h\|+\|\xi_h\|\leq \eps_h+\|\T_{\eps_h} \xi_h-\xi_h\|+\|\xi_h\|
	\leq \eps_h+3\|\xi_h\|\leq C.
	$$
	Hence, there exists $v\in X$ and a subsequence of $(v_h)$, that we will still indicate
	by $(v_h)$, such that $(v_h)$ weakly converges to $v$ in $X$. Since $X$ is symmetrically embedded into $W$,
	we have that $\|v_h^*\|\leq \|v_h\|\leq C$ and also, up to a further subsequence, $(v_h)$ converges in $W$ to some $\hat v\in W$.
	Of course, it is $v=\hat v$. If $f$ satisfies $(SQPS)_{\inf f}$,
	there exists a further subsequence, that we still denote by $(v_h)$, which converges to some $z$ in $X$. By lower semi-continuity, $f(z)=\inf f$.
	Since $|df|(v_h)\to 0$ and $f(v_h)\to \inf f=f(z)$, by means of \cite[Proposition 2.6]{dm}, it follows
	that $|df|(z)\leq\liminf_h |df|(v_h)=0$. Since $\|v_h-v_h^*\|_V\to 0$, letting $h\to\infty$ into 
	$\|z-z^*\|_V \leq \|z-v_h\|_V+\|v_h-v_h^*\|_V+\|v_h^*-z^*\|_V \leq K(C_\Theta+1)\|v_h-z\|+\|v_h-v_h^*\|_V$,  
	yields $z=z^*\in S$, as desired. Since $f(z)=\inf f$ and, by definition,
	$f(z+t\zeta)\geq f(z)$ and $f(z-t\zeta)\geq f(z)$ for all $t\in\R$ and $\zeta\in X$, 
	we infer that, for all $w \in X$,
	\begin{equation*}
	Q_z(w)\geq \limsup_{\substack{\zeta \to w \\ t\to 0}}\frac{f(z+t\zeta)+f(z-t\zeta)-2f(z)}{t^2}\geq 0.
	\end{equation*}
	This concludes the proof of the corollary.	
\end{proof}

\noindent
These results look particularly useful for applications to PDEs defined on a ball $\Omega$ or on $\R^N$, choosing 
$X=W^{1,p}_0(\Omega)$, $X=S$ or $S=W^{1,p}_0(\Omega,\R^+)$, $V=L^p\cap L^{p^*}(\Omega)$ and $W=L^q(\Omega)\supset V$
with $p<q<p^*$. These functional spaces are compatible with Definition~\ref{symmebeddddd}.

\subsection{Symmetric inf sup principles}
	\label{minimaxclass}
The symmetric version of Ekeland's variational principle 
allows to obtain a symmetric minimax type result, we refer to 
\cite[Theorem 5.1]{defig} for a standard statement without symmetry proved
via the standard Ekeland's principle. In fact,
let $(S,X,V,h,*)$ according to Definition~\ref{abssym}
and assume in addition that the map $\Theta:(X,\|\cdot\|)\to (S,\|\cdot\|)$ 
is continuous. Let $\psi\in S$ with $\psi^H=\psi$ 
for all $H\in {\mathcal H}_*$ (hence, in turn, $\psi^*=\psi$) and introduce the following spaces
\begin{align*}
& \hat X:=C([0,1],X),\quad \|\gamma\|_{\hat X}:=\sup_{t\in [0,1]}\|\gamma(t)\|, \quad
\hat V:=C([0,1],V)\quad \|\gamma\|_{\hat V}:=\sup_{t\in [0,1]}\|\gamma(t)\|_V, \\
& \hat S:=\{\gamma\in C([0,1],X):\gamma(0)=0,\,\,\,\gamma(1)=\psi\}.
\end{align*}
Define $*:\hat S\to \hat S$, $\gamma\mapsto \gamma^*$, 
and $h:\hat S\times {\mathcal H}_*\to \hat S$, $(\gamma,H)\mapsto \gamma^H$ by setting
$$
h(\gamma,H)(t):=\gamma(t)^H,\quad \gamma^*(t):=\gamma(t)^*,
\qquad\forall\gamma\in  \hat S,\,\,\forall H\in {\mathcal H}_*,\,\,\forall t\in [0,1].
$$
Notice that, since $X$ and $V$ are Banach spaces, $\hat X$ and $\hat V$ are Banach spaces, 
$\hat S\subset \hat X\subset \hat V$, $\hat S$ is a closed subset of $\hat X$ 
and $\hat X$ is continuously embedded into $\hat V$. Furthermore,
for all $\gamma\in \hat S$ it holds $\gamma^*,\gamma^H\in \hat S$
since $\gamma^*,\gamma^H\in C([0,1],X)$ and
$\gamma^H(0)=\gamma(0)^H=0$, $\gamma^H(1)=\psi^H=\psi$, 
$\gamma^*(0)=\gamma(0)^*=0$ and $\gamma^*(1)=\psi^*=\psi$.
It can be proved that $(\hat S,\hat X,\hat V,h,*)$ satisfies
the properties of Definition~\ref{abssym}. Given a $C^1$ functional
$f:X\to\R$ satisfying 
\begin{align}
\label{assumptionpol-001}
& f(u^H)\leq f(u),\quad\,\,\text{for all $u\in X$ and $H\in {\mathcal H}_*$}, \\
\label{boundbelow}
& \inf_{\gamma\in \hat S}\max_{t\in [0,1]} f(\gamma(t))> \max\{f(0),f(\psi)\},
\end{align}
consider the minimax value
$$
c=\inf_{\gamma\in \hat S}\max_{t\in [0,1]} f(\gamma(t)),
$$
and the functional $\hat f:\hat S\to\R$, bounded from below in view of \eqref{boundbelow}, defined by 
$$
\hat f(\gamma):=\max_{t\in [0,1]} f(\gamma(t)),\quad\text{for all $\gamma\in \hat S$.} 
$$
Notice that $\hat f$ is continuous (see \cite[proof of Theorem 5.1]{defig}) and, 
due to \eqref{assumptionpol-001}, it follows
$$
\hat f(\gamma^H)=\max_{t\in [0,1]} f(\gamma(t)^H)\leq 
\max_{t\in [0,1]} f(\gamma(t))=\hat f(\gamma),
\quad\text{for all $\gamma\in \hat S$.} 
$$
Then, by applying Theorem~\ref{ekelcor} (with $S'=\hat S$ and $\sigma=\rho=\eps>0$)
in place of the standard Ekeland's principle,
for every $\eps>0$ there exists $\gamma_\eps\in \hat S$ such that
$$
\|\gamma_\eps-\gamma_\eps^*\|_{\hat V}<\eps,\quad
c\leq \hat f(\gamma_\eps)\leq c+\eps,\quad
\hat f(\gamma)\geq \hat f(\gamma_\eps)-\eps \|\gamma-\gamma_\eps\|_{\hat X},\quad\forall \gamma\in \hat S.
$$
From these inequalities, by arguing along the lines of the proof of \cite[Theorem 5.1]{defig},
it is possible to show that, for every $\eps>0$ there exists $u_\eps\in X$
such that 
$$
\|u_\eps-u_\eps^*\|_V<\eps,\quad\,\,
\|df(u_\eps)\|\leq\eps,\quad\,\, c\leq f(u_\eps)\leq c+\eps. 
$$
The first inequality holds since, by construction, $u_\eps=\gamma_\eps(t_\eps)$
for some $t_\eps\in [0,1]$, yielding as desired
$\|u_\eps-u_\eps^*\|_V=\|\gamma_\eps(t_\eps)-\gamma_\eps(t_\eps)^*\|_V
\leq \|\gamma_\eps-\gamma_\eps^*\|_{\hat V}<\eps$.
Similar results were obtained in \cite{vansch} without using Ekeland variational principle.

\section{Some applications}
\label{applicationsprincip}

In this section we highlight possible applications of the abstract symmetric 
versions of the variational principles in the framework of PDEs, fixed point theory
and geometric properties of Banach spaces. 

\subsection{Calculus of variations} 
In this section we will consider two applications
of the symmetric principles to partial differential equations.

\subsubsection{A quasi-linear example}
\label{sectionPDEs}
Let $\Omega=B$ be the unit ball in $\R^N$ ($N\geq 1$), $1<p<\infty$ and define the functional
$f:W^{1,p}_0(\Omega)\to \R\cup\{+\infty\}$ by setting
\begin{equation}
	\label{f-calc-var}
f(u)=\int_{\Omega}{\mathcal L}(u,|Du|),
\end{equation}
where ${\mathcal L}$ is an integrand of class $C^1$ and, for $(s,\xi)\in\R\times\R^N$,
\begin{equation}
	\label{J1}
{\mathcal L}(s,|\xi|)\geq 0.
\end{equation}
Assume that $u$ belongs to $\dom{f}$ whenever $u\in W^{1,p}_0(\Omega)\cap L^\infty(\Omega)$.
The functions ${\mathcal L}_s$ and ${\mathcal L}_\xi$ are the derivatives 
of ${\mathcal L}$ with respect to the variables $s$ and $\xi$. We assume that there exist 
$\alpha,\beta,\gamma\in C(\R)$ and real numbers $a,b\in\R$ such that the following conditions hold: 
\begin{gather}
	\label{growthassumpptt0}
 {\mathcal L}(s,|\xi|)|\leq \alpha(|s|)|\xi|^p+b|\xi|^p+a,   \\
	\label{growthassumpptt}
 |{\mathcal L}_s(s,|\xi|)|\leq \beta(|s|)|\xi|^{p},\quad 
|{\mathcal L}_\xi(s,|\xi|)|\leq \gamma(|s|)|\xi|^{p-1}+b|\xi|^{p-1}+a,
\end{gather}
for every $(s,\xi)\in  \R\times\R^N$. We write the growth assumptions in such a fashion,
since in the particular case with $\beta=\gamma=0$, conditions \eqref{J1}-\eqref{growthassumpptt} reduce to
\cite[assumptions (4.12), (4.13) and (4.14)]{ekeland1} stated by Ekeland.
Now, since in the general case where $\beta$ and $\gamma$ are unbounded, ${\mathcal L}_s(u,|Du|)$
and ${\mathcal L}_\xi(u,|Du|)$ are not in $L^1_{{\rm loc}}(B)$ for a given function $u\in W^{1,p}_0(\Omega)$,
the Euler-Lagrange equation associated with $f$ cannot be given, at least a priori, a distributional sense.
To overcome this situation, in \cite{pelsqu}, for every $u\in W^{1,p}_0(\Omega)$ the following vector space, 
dense in $W^{1,p}_0(\Omega)$, was used
\begin{equation}\label{defvu}
V_u=\left\{v\in W^{1,p}_0(\Omega)\cap L^\infty(\Omega):\,
u\in L^{\infty}(\{x\in\Omega:\,v(x)\not=0\})\right\}.
\end{equation}
The following proposition can be obtained arguing as in \cite[Proposition 4.5]{pelsqu} and provides a link between
the weak slope and directional derivatives of $f$ along a direction $v\in V_u$.

\begin{proposition}\label{stimadj}
Under assumptions~\eqref{J1}-\eqref{growthassumpptt}, for every $u\in\dom{f}$, we have
\begin{equation*}
|df|(u)\geq \sup_{\substack{v\in V_u \\ \|v\|_{1,p}\leq 1}} 
\Big[\int_\Omega  {\mathcal L}_\xi(u,|Du|)\cdot Dv+\int_\Omega  {\mathcal L}_s(u,|Du|)v\Big].
\end{equation*}
\end{proposition}

\noindent
As a consequence of Proposition~\ref{stimadj} and Theorem~\ref{ekelcorII}, we have the following

\begin{theorem}
	\label{Eksymnonsm}
Assume that conditions \eqref{J1}-\eqref{growthassumpptt} hold and 
${\mathcal L}(-s,|\xi|)\leq {\mathcal L}(s,|\xi|)$ for all $s\leq 0$. Then, for any $\eps>0$, there exist 
$u_\eps\in W^{1,p}_0(\Omega)$ and $w_\eps\in W^{-1,p'}(\Omega)$ such that
\begin{equation}
	\label{quasisoll}
\langle w_\eps, v\rangle=\int_\Omega  {\mathcal L}_\xi(u_\eps,|Du_\eps|)\cdot Dv+
\int_\Omega  {\mathcal L}_s(u_\eps,|Du_\eps|)v\quad\,\,\forall v\in V_{u_\eps},
\end{equation}
as well as
\begin{equation*}
\|w_\eps\|_{W^{-1,p'}(\Omega)}\leq\eps
\quad
\text{and}
\quad
\|u_\eps-u_\eps^*\|_{L^p(\Omega)\cap L^{p^*}(\Omega)}<\eps.
\end{equation*}
\end{theorem}
\begin{proof}
	The functional $f$ in formula \eqref{f-calc-var} is proper, bounded below and lower semi-continuous by means of
	condition \eqref{J1} and Fatou's lemma.
	Moreover, the assumptions of Theorem~\ref{ekelcorII} are satisfied with $X=W^{1,p}_0(\Omega)$,
	$S=W^{1,p}_0(\Omega,\R^+)$, $V=L^p(\Omega)\cap L^{p^*}(\Omega)$, $\xi=|u|$ and 
	where $u^H,u^*$ for $u\in S$ (and $u^*=|u|^*$ for $u\in X$) denote
	the polarization and symmetrization (see sections \ref{concresectpol}-\ref{concresectsym}). Assumption \eqref{assumptionpol-00} 
	holds with equal sign by the radial structure of the integrand, as it can be verified
	by direct computation. The assertion follows by Theorem~\ref{ekelcorII} (recall also Remark	\ref{weakstrong}), 
	Proposition~\ref{stimadj} and the Hahn-Banach theorem, taking into account the density of $V_{u_\eps}$ in $W^{1,p}_0(\Omega)$.
\end{proof}

\noindent
In many cases one recovers the fact that the solution $u_\eps$ of equation 
\eqref{quasisoll} is actually meant in the sense of distributions, by suitably enlarging the class
of admissible test functions, see e.g.~\cite[Theorem 4.10 and Lemma 4.6]{pelsqu}.
Theorem~\ref{Eksymnonsm} could be seen as a non-smooth symmetric version of \cite[Proposition 4.3(a)]{ekeland1}.
In fact, under the above assumptions our functional is merely lower semi-continuous, while the
functional of \cite[Proposition 4.3(a)]{ekeland1} is of class $C^1$. Furthermore, the symmetry features in Theorem~\ref{Eksymnonsm}
can be obtained via Theorem~\ref{ekelcorII} due to the structure ${\mathcal L}(s,|\xi|)$ yielding \eqref{assumptionpol-00}, 
in place of the more general form 
${\mathcal L}(x,s,\xi)$, admissible in \cite{ekeland1}. Theorem~\ref{Eksymnonsm} is new even in the particular case
$\beta=\gamma=0$. We stress that a constrained version of Theorem~\ref{Eksymnonsm} could also be provided, yielding a 
non-smooth symmetric counterpart of \cite[Proposition 4.3(b)]{ekeland1}.

\subsubsection{A semi-linear example}
Let us now briefly discuss another example where the second 
order condition related to $w\mapsto Q_u(w)$ is also involved,
namely the inferior limit in formula~\eqref{addproperty}, in a $C^1$, but not $C^2$, framework.
In \cite{bartschdeg}, Bartsch and Degiovanni showed that, in some concrete cases of interest
in the theory of partial differential equations, although it is often not possible to compute the values of $Q_u(w)$,
it is possible to compute a greater quantity greater. For instance, if $f$ is of class $C^1$, 
then, see \cite[Remark 4.4]{bartschdeg}, for every $w\in X$
$$
Q_u(w)\leq \limsup_{\substack{(\tau,\theta)\to (0,0) \\ z\to u \\ \zeta\to w}}\frac{f'(z+\tau\zeta)\zeta-f'(z+\theta\zeta)\zeta}{\tau-\theta},
$$ 
being the right-hand side more easy to estimate, in some cases \cite[Propositions 4.5]{bartschdeg}.
For instance, let now $\Omega=B$ be the unit ball in $\R^3$, the three dimensional case being considered just for simplicity. 
Let also $g:\R\to\R$ be a continuous function and assume that there exist $a_1,a_2\in\R$, $b\in\R$
and $2<p\leq 6$ such that, for all $s,t\in\R$, it holds
\begin{gather*}
 |g(s)| \leq a_1+b|s|^{p-1}\quad\,\text{and}\,\quad g(-s)=-g(s), \\
 (g(s)-g(t))(s-t) \geq -(a_2+b|s|^{p-2}+b|t|^{p-2})(s-t)^2.
\end{gather*}
Then, for all $s\in\R$, define a measurable function $\underline{D}_s g$ by setting
$$
\underline{D}_s g(s):=\liminf_{\substack{(t,\tau)\to (0,0) \\ t,\tau\in\Q}}\frac{g(s+t)-g(s+\tau)}{t-\tau}.
$$
Let $G(s)=\int_0^s g(t)dt$ and consider the $C^1$ functional $f:H^1_0(\Omega)\to\R$ defined by
$$
f(u)=\frac{1}{2}\int_\Omega |Du|^2-\int_\Omega G(u).
$$
In light of~\cite[Proposition 6.1]{bartschdeg}, it holds
\begin{equation}
	\label{inequality}
Q_u(w)\leq \int_\Omega |Dw|^2-\int_\Omega \underline{D}_s g(u)w^2<+\infty,\quad
\forall u,w\in H^1_0(\Omega). 
\end{equation}

\noindent
Therefore, combining Theorem~\ref{maincor} with the above setting yields the following
\begin{theorem}
	\label{EksymnonsmOO}
Assume that $f$ is bounded from below and admits a bounded minimizing sequence.
Then $f$ has a minimizing sequence $(u_h)\subset H^1_0(\Omega)$ and a 
sequence $(\psi_h)\subset H^{-1}(\Omega)$ such that
\begin{align*}
& \lim_h \|u_h-u_h^*\|_{L^2(\Omega)\cap L^{2^*}(\Omega)}=0, \\
& \int_\Omega Du_hD\varphi =\int_\Omega g (u_h)\varphi+\langle \psi_h,\varphi\rangle,\,\,\quad \forall\varphi\in H^1_0(\Omega),
\quad \lim_h\|\psi_h\|_{H^{-1}}=0, \\
& \liminf_h\Big[\int_\Omega |Dw|^2-\int_\Omega \underline{D}_s g(u_h)w^2\Big]\geq 0,\quad\,\, \forall w\in H^1_0(\Omega).
\end{align*}
\end{theorem}
\begin{proof}
Based upon the above remarks, the assertion follows by Theorem~\ref{maincor} by choosing $X=H^1_0(\Omega)$, $S=H^1_0(\Omega,\R^+)$,
$V=L^2(\Omega)\cap L^{2^*}(\Omega)$, since
$f(u^H)=f(u)$ for all $u\in S$ and $H\in {\mathcal H}_*$, as well as $f(|u|)=f(u)$
for all $u\in X$ and $\|u^H\|_{H^1_0(\Omega)}=\|u\|_{H^1_0(\Omega)}$ for all 
$u\in S$ and $H\in {\mathcal H}_*$.
\end{proof}

\subsection{Fixed points}

The following is a symmetric version of the so-called Caristi Fixed Point Theorem \cite{caristi},
that was also proved by Ekeland via his principle in \cite{ekeland2}.

\begin{theorem}[{Symmetric Caristi Fixed Point Theorem}]
	\label{caristi}
Let $X$ be a Banach space and $F:X\to X$ a map such that
$$
\|F(u)-u\|\leq f(u)-f(F(u)),\,\,\quad\text{for all $u\in X$},
$$
where $f:X\to\R$ is a lower semi-continuous function, bounded from below,
satisfying~\eqref{assumptionpol-00} and such that, for all $u\in X$ there exists 
$\xi\in S$ with $f(\xi)\leq f(u)$. Then, for all $\eps\in (0,1)$, there exists a fixed point
$\xi_\eps\in X$ of $F$ such that $\|\xi_\eps-\xi_\eps^*\|_V<\eps$.
\end{theorem}
\begin{proof}
By virtue of Theorem~\ref{ekelcorII} with $\sigma=\rho=\eps>0$, for every $\eps\in (0,1)$ there exists an element 
$\xi_\eps\in X$ such that $\|\xi_\eps-\xi_\eps^*\|_V<\eps$, and 
$$
f(w)\geq f(\xi_\eps)-\eps\|w-\xi_\eps\|,\qquad\text{for all $w\in X$}.
$$
In particular, choosing $w=F(\xi_\eps)$ and using the assumption, we get
$$
\|F(\xi_\eps)-\xi_\eps\|\leq f(\xi_\eps)-f(F(\xi_\eps))
\leq \eps\|F(\xi_\eps)-\xi_\eps\|,
$$
which yields $F(\xi_\eps)=\xi_\eps$, concluding the proof.
\end{proof}

\noindent
Let $\Omega$ be either a ball in $\R^N$ of the whole $\R^N$ and $1<p<\infty$. 

\begin{corollary}
Let $F:W^{1,p}_0(\Omega)\to W^{1,p}_0(\Omega)$ a map such that
$$
\|F(u)-u\|_{1,p}\leq f(u)-f(F(u)),\qquad\text{for all $u\in W^{1,p}_0(\Omega)$},
$$
where $f:W^{1,p}_0(\Omega)\to\R$ is a lower semi-continuous function bounded from below such that
$$
f(|u|)\leq f(u)\quad\text{for all $u\in W^{1,p}_0(\Omega)$},\qquad
f(u^H)\leq f(u)\quad\text{for all $u\in W^{1,p}_0(\Omega,\R^+)$}.
$$
Then, for all $\eps\in (0,1)$, there is a fixed point
$\xi_\eps\in W^{1,p}_0(\Omega)$ of $F$ with $\|\xi_\eps-\xi_\eps^*\|_{L^p\cap L^{p^*}(\Omega)}<\eps$.
\end{corollary}
\begin{proof}
Theorem~\ref{caristi} is applied with $X=W^{1,p}_0(\Omega)$, $S=W^{1,p}_0(\Omega,\R^+)$ 
and $V=L^p\cap L^{p^*}(\Omega)$. As pointed out on Section~\ref{polarizationsect}, if $u^H$ is the polarization of positive functions on $\R^N$ 
and $*$ is the Schwarz symmetrization, the framework of Definition~\ref{abssym} is satisfied.	
\end{proof}

\noindent
Let $\Omega$ be either a ball in $\R^N$ of the whole $\R^N$ and $1<p<\infty$.

\begin{corollary}
Let $F:L^p(\Omega)\to L^p(\Omega)$ a map such that
$$
\|F(u)-u\|_p\leq f(u)-f(F(u)),\qquad\text{for all $u\in L^p(\Omega)$},
$$
where $f:L^p(\Omega)\to\R$ is a lower semi-continuous function bounded from below such that
$$
f(|u|)\leq f(u)\quad\text{for all $u\in L^p(\Omega)$},\qquad
f(u^H)\leq f(u)\quad\text{for all $u\in L^p(\Omega,\R^+)$}.
$$
Then, for all $\eps\in (0,1)$, there is a fixed point
$\xi_\eps\in L^p(\Omega)$ of $F$ with $\|\xi_\eps-\xi_\eps^*\|_{L^p(\Omega)}<\eps$.
\end{corollary}
\begin{proof}
Theorem~\ref{caristi} is applied with $X=V=L^p(\Omega)$ and $S=L^p(\Omega,\R^+)$. 
\end{proof}

\noindent
We conclude the section with a symmetric version of a fixed point theorem due to Clarke \cite{clarke} and
also proved by Ekeland via his principle \cite{ekeland2}.

\begin{theorem}
	\label{clarkethm}
	Let $(X,\|\cdot\|_V)$ be a Banach space, $F:(X,\|\cdot\|_V)\to (X,\|\cdot\|_V)$ continuous 
	and assume that there exists $0<\sigma<1$ such that
	\begin{equation}
		\label{clarkeasss}
	\forall u\in X\,\, \exists t\in (0,1]:\,\,\,
	\|F(t F(u)+(1-t)u)-F(u)\|_V\leq \sigma t\|F(u)-u\|_V.
	\end{equation}
Assume that $F(S)\subset S$, $F(u^H)=F(u)^H$ for all $H\in {\mathcal H}_*$ and $u\in S$,
and that for every $u\in X$ there exists $\xi\in S$ such that $\|\xi-F(\xi)\|_V\leq \|u-F(u)\|_V$.
Then, for any $\eps\in (0,1-\sigma)$ there exists a fixed point $\xi_\eps\in X$ for $F$ 
such that $\|\xi_\eps-\xi_\eps^*\|_V<\eps$.
\end{theorem}
\begin{proof}
It is sufficient to argue essentially as in the proof of 
\cite[Theorem 3]{ekeland2} on the function $f:X\to\R$ defined by
$f(u):=\|u-F(u)\|_V$ observing that, by assumption and by (5) of Definition~\ref{abssym},
it holds $f(u^H)=\|u^H-F(u^H)\|_V=\|u^H-F(u)^H\|_V\leq \|u-F(u)\|_V=f(u)$ 
for all $H\in {\mathcal H}_*$ and $u\in S$.
Moreover, for all $u\in X$ there is $\xi\in S$ such that $f(\xi)\leq f(u)$. 
Applying Theorem~\ref{ekelcorII} in place of Ekeland's principle,
the assertion follows.
\end{proof}	

\noindent
Let $\Omega$ be either a ball in $\R^N$ of the whole $\R^N$ and $1<p<\infty$.
\begin{corollary} 
Let $F:L^p(\Omega)\to L^p(\Omega)$ a map such that \eqref{clarkeasss} holds, 
$F(u)\geq 0$ for all $u\in L^p(\Omega,\R^+)$, $F(u^H)=F(u)^H$ for all $H\in {\mathcal H}_*$ and $u\in L^p(\Omega,\R^+)$,
and that $F(|u|)=|F(u)|$ for all $u\in L^p(\Omega)$.
Then, for every $\eps\in (0,1-\sigma)$ there is a fixed point $\xi_\eps\in L^p(\Omega)$ for $F$ 
such that $\|\xi_\eps-\xi_\eps^*\|_{L^p(\Omega)}<\eps$
\end{corollary}
\begin{proof}
Apply Theorem~\ref{clarkethm} with the choice $X=V=L^p(\Omega)$, $S=L^p(\Omega,\R^+)$. Notice that, for
all $u\in X$, it holds $\||u|-F(|u|)\|_{L^p(\Omega)}=\||u|-|F(u)|\|_{L^p(\Omega)}\leq \|u-F(u)\|_{L^p(\Omega)}$,
in the notations of the proof of Theorem~\ref{clarkethm}.
\end{proof}

\subsection{Drops and flower-petals}

As a by-product of the symmetric variational principle, Theorem~\ref{ekelcor},  
we obtain symmetric versions of Dane\u{s} Drop Theorem \cite{danes} and of the 
Flower Petal Theorem \cite{penot}. In the particular case where $h$ and $*$ are the identity
maps and $S=X=V$, then the statements reduce to the classical formulation.
Possible applications of the statements in some meaningful concrete situations
have not yet been investigated.

\begin{definition}
Let $X$ be a Banach space, $B\subset X$ convex and $x\in X$. We say that
$$
{\rm Drop}(x,B):=\bigcup_{y\in B,\, t\in [0,1]} x+t(y-x),
$$
is the {\rm drop} associated with $x$ and $B$. If $x_0,x_1\in X$ and $\eps>0$, we say that
$$
{\rm Petal}_\eps(x_0,x_1):=\big\{y\in X:\eps\|y-x_0\|+\|y-x_1\|\leq \|x_0-x_1\|\big\}
$$
is the {\rm petal} associated with $\eps$ and $x_0,x_1\in X$. 
\end{definition}

\noindent
Notice that, for all $\eps\in (0,1)$ and $x_0,x_1\in X$ it always holds
$$
B_{\frac{1-\eps}{1+\eps}\|x_0-x_1\|}(x_1)\subset  {\rm Petal}_\eps(x_0,x_1),\quad 
{\rm Drop}\big(x_0,B_{\frac{1-\eps}{1+\eps}\|x_0-x_1\|}(x_1)\big)\subset  {\rm Petal}_\eps(x_0,x_1),
$$
so that each petal contains a suitable ball as well as a drop of a suitable ball.

\vskip2pt
\noindent
Here is a symmetric version of the so called {\em Drop} Theorem due to Dane\u{s}~\cite{danes}.

\begin{theorem}[{Symmetric Drop Theorem}]
	\label{drop}
Let $(X,\|\cdot\|_V)$ be a Banach space, $B,C$ nonempty closed subsets of $S$ with $B\subset X_{{{\mathcal H}_*}}$ convex and 
$d(B,C)>0$. Moreover, let $x\in C$ such that $S':={\rm Drop}(x,B)\cap C$ is 
closed and $h(S')\subset S'$, $*(S')\subset S'$. Then, for all $\eps>0$ small, there exists 
$\xi_\eps\in {\rm Drop}(x,B)\cap C$ such that
$$
{\rm Drop}(\xi_\eps,B)\cap C=\{\xi_\eps\}
\quad
\text{and}
\quad
\|\xi_\eps-\xi^*_\eps\|_V< \eps.
$$
\end{theorem}
\begin{proof}
By Remark~\ref{restriction}, $(S',X,V,h,*)$ satisfies (1)-(5) of Definition~\ref{abssym}
and Proposition~\ref{mapJvS}. Moreover, $S'$ is closed. Define a continuous 
function $f:S'\to\R^+$ by setting
$$
f(u):=\inf_{\zeta\in B}\|u-\zeta\|_V, \,\,\quad \text{for all $u\in S'$}.
$$
Observe that, since $B\subset X_{{{\mathcal H}_*}}$, for all $u\in S'$ and any $H\in {\mathcal H}_*$, we have
$$
f(u^H)=\inf_{\zeta\in B}\|u^H-\zeta\|_V=\inf_{\zeta\in B}\|u^H-\zeta^H\|_V\leq
\inf_{\zeta\in B}\|u-\zeta\|_V=f(u),
$$
in light of (5) of Definition~\ref{abssym}.
Let now $\eps_0>0$ be fixed sufficiently small that $\eps_0\, {\rm diam}(B)<(1-\eps_0)d(B,C)$. 
In turn, for every $\eps\in(0,\eps_0]$, by applying Theorem~\ref{ekelcor} 
with $\rho=\sigma=\eps$, we find an element $\xi_\eps\in S'$ such that $\|\xi_\eps-\xi^*_\eps\|_V<\eps$ and 
\begin{equation}
	\label{variationrel}
\inf_{\zeta\in B}\|w-\zeta\|_V>\inf_{\zeta\in B}\|\xi_\eps-\zeta\|_V-\eps\|w-\xi_\eps\|_V,
\quad\forall w\in S\setminus\{\xi_\eps\}.
\end{equation}
To prove the assertion, we argue by contradiction, assuming that
$$
{\rm Drop}(\xi_\eps,B)\cap ({\rm Drop}(x,B)\cap C)\neq \{\xi_\eps\}.
$$
Then, we find $\tau\in [0,1]$, $\tau\neq 1$, and $\eta\in B$ such that $\hat w=(1-\tau) \eta+\tau\xi_\eps\in S'\setminus\{\xi_\eps\}$.
In turn, from formula \eqref{variationrel} evaluated at $\hat w$, and since $B$ is convex, we infer
\begin{equation*}
\inf_{\zeta\in B}\|\xi_\eps-\zeta\|_V <\tau\inf_{\zeta\in B}\|\xi_\eps-\zeta\|_V
+(1-\tau)\inf_{\zeta\in B}\|\eta-\zeta\|_V+\eps(1-\tau)\|\eta-\xi_\eps\|_V,
\end{equation*}
namely (recall that $0\leq\tau<1$) for every $\zeta\in B$ it holds
\begin{equation*}
\inf_{\zeta\in B}\|\xi_\eps-\zeta\|_V <\eps\|\eta-\xi_\eps\|_V\leq \eps\, {\rm diam}(B)+\eps\|\zeta-\xi_\eps\|_V.
\end{equation*}
Therefore, taking the infimum over $\zeta\in B$, and since $\eps\in (0,\eps_0]$, we conclude that
\begin{equation*}
(1-\eps_0)d(B,C)\leq (1-\eps)\inf_{\zeta\in B}\|\xi_\eps-\zeta\|_V \leq \eps\, {\rm diam}(B)
\leq  \eps_0\, {\rm diam}(B)<(1-\eps_0)d(B,C),
\end{equation*}
namely a contradiction. Hence, 
${\rm Drop}(\xi_\eps,B)\cap ({\rm Drop}(x,B)\cap C)=\{\xi_\eps\}$. 
By the inclusion ${\rm Drop}(\xi_\eps,B)\subset {\rm Drop}(x,B)$
we get ${\rm Drop}(\xi_\eps,B)\cap C=\{\xi_\eps\}$, concluding the proof.
\end{proof}

\noindent
Let now $\Omega$ be either the unit ball in $\R^N$ or $\R^N$ and $1<p<\infty$.
We denote by $L^p_r(\Omega,\R^+)$ the set of radially symmetric 
elements of $L^p(\Omega,\R^+)$, that is $u^*=u$, being $*$ the 
Schwarz symmetrization, being equivalent to $u^H=u$ for any $H\in {\mathcal H}_*$.

\begin{corollary}[{Symmetric Drop Theorem in $L^p$-spaces}]
	\label{dropSobolev}
Let $C$ be a nonempty closed subset of $(L^p(\Omega,\R^+),\|\cdot\|_{L^p(\Omega)})$ and 
$B$ a unit ball in $L^p_r(\Omega,\R^+)$ with $d(B,C)>0$.
Let $u\in C$ be such that
\begin{align*}
\forall v\in {\rm Drop}(u,B)\cap C,\,\,\forall  H\in {\mathcal H}_*:&\,\,\,\,   v^H\in {\rm Drop}(u,B)\cap C, \\
\forall v\in {\rm Drop}(u,B)\cap C:&\,\,\,\,   v^*\in {\rm Drop}(u,B)\cap C.
\end{align*}
Then, for all $\eps>0$ small, there exists 
$\xi_\eps\in {\rm Drop}(u,B)\cap C$ such that
$$
{\rm Drop}(\xi_\eps,B)\cap C=\{\xi_\eps\}
\quad
\text{and}
\quad
\|\xi_\eps-\xi^*_\eps\|_{L^p(\Omega)}<\eps.
$$
\end{corollary}
\begin{proof}
By assumption, $S'$ is compatible with Definition~\ref{abssym}.
Apply Theorem~\ref{drop} with $X=V=L^p(\Omega)$, $S=L^p(\Omega,\R^+)$, $S'={\rm Drop}(u,B)\cap C$. 
Since $B\subset L^p_r(\Omega,\R^+)$, $u^*=u$ for all $u\in B$ and thus
$u^H=u$ for all $H\in {\mathcal H}_*$.
Hence, $B$ is a convex subset of $X_{{{\mathcal H}_*}}$. 
\end{proof}

\noindent
Next, we state a symmetric version of the {\em Petal Flower} Theorem obtained by Penot \cite{penot}.

\begin{theorem}[{Symmetric Petal Flower Theorem}]
	\label{petal}
Let $(X,\|\cdot\|_V)$ be a Banach space, $S'=C$ a closed subset of $S$ such that 
\begin{align*}
\forall v\in C,\,\,\forall  H\in {\mathcal H}_*:&\,\,\,\,   v^H\in C, \\
\forall v\in C:&\,\,\,\,   v^*\in C.
\end{align*}
Assume that $x\in C$, $y\in S\setminus C$ with 
$x^H=x$ and $y^H=y$ for any $H\in {\mathcal H}_*$ and 
\begin{equation}
	\label{assumptyC}
\|x-y\|_V\leq d(y,C)+\eps^2,\,\,\quad\text{for some $\eps>0$.}
\end{equation}
Then there exists a point $\xi_\eps\in {\rm Petal}_\eps(x,y)\cap C$ such that
$$
{\rm Petal}_\eps(\xi_\eps,y)\cap C=\{\xi_\eps\}
\quad
\text{and}
\quad
\|\xi_\eps-\xi^*_\eps\|_V<\eps.
$$
\end{theorem}
\begin{proof}
By Remark~\ref{restriction}, $(S',X,V,h,*)$ satisfies (1)-(5) of Definition~\ref{abssym}
and Proposition~\ref{mapJvS}. Moreover, $S'$ is closed.	
Define the continuous map $f:S'\to\R^+$ 
by setting $f(u):=\|u-y\|_V$ for all $u\in S'$. 
Since $y^H=y$ for any $H\in {\mathcal H}_*$, we have
$$
f(u^H)=\|u^H-y\|_V=\|u^H-y^H\|_V\leq\|u-y\|_V
=f(u),\quad\text{for $u\in S'$ and $H\in {\mathcal H}_*$}. 
$$
Then, by Theorem~\ref{ekelcorV} and Remark~\ref{ekelcorV-rmk}, with the choice $\rho=\sigma=\eps$, 
since~\eqref{assumptyC} rephrases as $f(x)\leq \inf_{S'} f+\eps^2$, there exists
$\xi_\eps\in C$ such that $\|\xi_\eps-\xi_\eps^*\|_V<\eps$,
$$
\eps\|w-\xi_\eps\|_V+\|w-y\|_V>\|\xi_\eps-y\|_V, \,\,\quad \forall w\in C\setminus\{\xi_\eps\}, 
$$
and 
$\eps\|\xi_\eps-\T_\eps x\|_V+\|\xi_\eps-y\|_V\leq \|x-y\|_V$. As $\T_\eps x=x$, this means $\xi_\eps\in {\rm Petal}_\eps (x,y)\cap C$
and $w\not\in {\rm Petal}_\eps (\xi_\eps,y)$ for all $w\in C\setminus\{\xi_\eps\}$, that is ${\rm Petal}_\eps(\xi_\eps,y)\cap C=\{\xi_\eps\}$.
\end{proof}

\noindent
Let now $\Omega$ be either the unit ball in $\R^N$ or the whole $\R^N$ and take $1\leq p<\infty$.

\begin{corollary}[{Symmetric Petal Flower Theorem in $L^p$-spaces}]
	\label{petalSobolev}
Let $C$ be a closed subset of $(L^p(\Omega,\R^+),\|\cdot\|_{L^p(\Omega)})$, $u\in C$, 
$v\in L^p(\Omega,\R^+)\setminus C$ with $u^H=u$ and $v^H=v$ for any $H\in {\mathcal H}_*$,
$\|u-v\|_{L^p(\Omega)}\leq d(v,C)+\eps^2$ for some $\eps>0$. Assume in addition that
\begin{align*}
\forall v\in C,\,\,\forall  H\in {\mathcal H}_*:&\,\,\,\,   v^H\in C, \\
\forall v\in C:&\,\,\,\,   v^*\in C.
\end{align*}
Then there exists $\xi_\eps\in {\rm Petal}_\eps(u,v)\cap C$ with ${\rm Petal}_\eps(\xi_\eps,v)\cap C=\{\xi_\eps\}$ and
$\|\xi_\eps-\xi^*_\eps\|_{L^p(\Omega)}< \eps$.
\end{corollary}
\begin{proof}
Apply Theorem~\ref{petal}, with the choice $X=V=L^p(\Omega)$ and $S=L^p(\Omega,\R^+)$. 
\end{proof}


\bigskip
\medskip


\bigskip

\begin{thebibliography}{99}
	
\bibitem{aubin-ekeland}
{\sc J.-P.\ Aubin, I.\ Ekeland}, 
Applied nonlinear analysis. Pure and Applied Mathematics,
Wiley, New York, 1984.	

\bibitem{bartschdeg}
{\sc T.\ Bartsch, M.\ Degiovanni},
Nodal solutions of nonlinear elliptic Dirichlet problems on radial domains,
{\em Atti Accad. Naz. Lincei Cl. Sci. Fis. Mat. Natur. (9) Mat. Appl.} {\bf 17} (2006), 69-85.
	
\bibitem{borweinpreiss}	
{\sc J.M.~Borwein, D.~Preiss}, 
A smooth variational principle with applications to 
subdifferentiability and to differentiability of convex 
functions, {\em Trans. Amer. Math. Soc.} {\bf 303} (1987), 517-527.

\bibitem{borwzhu}
{\sc J.M.\ Borwein, Q.J.\  Zhu}, 
Techniques of variational analysis, 
CMS Books in Mathematics {\bf 20}, Springer, 2005. 

\bibitem{cakliwill}
{\sc L.\ Caklovi\'c,  S.J.\ Li, M.\ Willem}, 
A note on Palais-Smale condition and coercivity,
{\em Differential Integral Equations} {\bf 3} (1990), 799-800.

\bibitem{caristi}
{\sc J.\ Caristi}, 
Fixed point theorems for mappings satisfying inwardness conditions,
{\em Trans. Amer. Math. Soc.} {\bf 215} (1976), 241-251.	

\bibitem{cerami}
{\sc G. Cerami}, 
An existence criterion for the critical points on unbounded manifolds,
{\em Istit. Lombardo Accad. Sci. Lett. Rend. A} {\bf 112} (1978), 332-336.

\bibitem{clarke}
{\sc F.H.\ Clarke}, 
Pointwise contraction criteria for the existence of fixed points,
{\em Canad. Math. Bull.} {\bf 21} (1978), 7-11.

\bibitem{corvel}
{\sc J.-N. Corvellec},
A note on coercivity of lower semicontinuous functions and nonsmooth critical point theory,
{\em Serdica Math. J.} {\bf 22} (1996), 57–68.

\bibitem{danes}
{\sc J.\ Dane\u{s}}, 
A geometric theorem useful in nonlinear functional analysis,
{\em Boll. UMI} {\bf 6} (1972), 369-375.	

\bibitem{defig}
{\sc D.G.\ de Figueiredo},
Lectures on the Ekeland variational principle with applications and detours. 
Tata Institute of Fundamental Research Lectures Math and Phys, {\bf 81}. 
Springer-Verlag, Berlin, 1989.
	
\bibitem{dm}
{\sc M.\ Degiovanni, M.\ Marzocchi},
A critical point theory for nonsmooth functionals,
{\em Ann.\ Mat.\ Pura Appl.} {\bf 167} (1994), 73-100.

\bibitem{degiorgimarinotosq}
{\sc E.\ De Giorgi, A.\ Marino, M.\ Tosques}, 
Problems of evolution in metric spaces and maximal decreasing curve,
{\em Atti Accad. Naz. Lincei Rend. Cl. Sci. Fis. Mat. Natur. (8)} {\bf 68} (1980), 180-187.

\bibitem{deville}
{\sc R. Deville, G.\ Godefroy, V.\ Zizler},
A smooth variational principle with applications to Hamilton-Jacobi equations in infinite dimensions,
{\em J. Funct. Anal.} {\bf 111} (1993), 197-212.

\bibitem{kadek}
{\sc J.\ Diestel}, 
Geometry of Banach spaces,
{\em Lecture Notes Math} {\bf 485} Springer Verlag, 1975.

\bibitem{ekeland1}
{\sc I.\ Ekeland}, 
On the variational principle, 
{\em J. Math. Anal. Appl.} {\bf 47} (1974), 324-353.

\bibitem{ekeland2}
{\sc I.\ Ekeland}, 
Nonconvex minimization problems, 
{\em Bull. Amer. Math. Soc.} {\bf 1} (1979), 443-474.

\bibitem{ekelandbook}
{\sc I.\ Ekeland}, 
Convexity methods in Hamiltonian mechanics,
{\bf 19} Springer-Verlag, Berlin, 1990. 

\bibitem{ekelandtemam}
{\sc I.\ Ekeland, R.\ Temam}, 
Convex analysis and variational problems, Amsterdam, 1976. 

\bibitem{fanghou}
{\sc G.\ Fang, N.\ Ghoussoub}, 
Second-order information on Palais Smale sequences in the mountain pass theorem, 
{\em Manuscripta Math.} {\bf 75} (1992), 81-95.

\bibitem{georgiev}
{\sc P.G.\ Georgiev}, 
The strong Ekeland variational principle, the strong drop theorem and applications, 
{\em J. Math. Anal. Appl.} {\bf 131} (1988), 1-21.

\bibitem{ghoubook}
{\sc N.\ Ghoussoub}, 
Duality and perturbation methods in critical point theory,
Cambridge Tracts in Mathematics {\bf 107}, Cambridge University Press, Cambridge, 1993.

\bibitem{lionsym}
{\sc P.-L.\ Lions},
Sym\'etrie et compacit\'e dans les espaces de Sobolev,
{\em J. Funct. Anal.} {\bf 49} (1982), 315-334.

\bibitem{lionscmp}
{\sc P.-L.\ Lions}, 
Solutions of Hartree-Fock equations for Coulomb systems, 
{\em Comm. Math. Phys.} {\bf 109} (1987), 33-97. 

\bibitem{mawhwill}
{\sc J.\ Mawhin, M.\ Willem},
Origin and evolution of the Palais-Smale condition in critical point theory,
{\em J.\ Fixed Point Theory Appl.} {\bf 7} (2010), 265-290.

\bibitem{pelsqu}
{\sc B.\ Pellacci, M.\ Squassina}, 
Unbounded critical points for a class of lower semicontinuous functionals, 
{\em J. Differential Equations} {\bf 201} (2004), 25-62.

\bibitem{penot}
{\sc J.-P.\ Penot}, 
The drop theorem, the petal theorem and Ekeland's variational principle,
{\em Nonlinear Anal.\ } {\bf 10} (1986), 813-822.

\bibitem{vansch}
{\sc J.\ Van Schaftingen}, 
Symmetrization and minimax principles,
{\em Comm.\ Contemp.\ Math.} {\bf 7} (2005), 463-481.

\bibitem{squass}
{\sc M. Squassina}, 
On Ekeland's variational principle, preprint, 4 pages.

\bibitem{strauss}
{\sc W.\ Strauss}, 
Existence of solitary waves in higher dimensions, 
{\em Comm. Math. Phys.} {\bf 55} (1977), 149-162.

\bibitem{suzuki}
{\sc T. Suzuki}. On the relation between the weak Palais-Smale condition and coercivity given by Zhong, 
{\em Nonlinear Anal.} {\bf 68} (2008), 2471-2478.

\bibitem{zhong}
{\sc C.-K. Zhong}, 
On Ekeland's variational principle and a minimax theorem, 
{\em J. Math. Anal. Appl.} {\bf 205} (1997), 239-250.

\bibitem{willem}
{\sc M.\ Willem}, 
Minimax theorems. 
Progress in Nonlinear Differential Equations and their Applications, {\bf 24} Birkh\"auser, Boston, 1996.

\end{thebibliography}
\end{document}